\pdfoutput=1
\documentclass[toc]{mathprint} %
\usepackage{rutarmacros}
\usepackage{macros}

\title[Attainable lower spectra]{Attainable forms of lower spectra}

\author
  {Amlan Banaji}
  {Department of Mathematics and Statistics,
      University of Jyväskylä,
      Agora, Mattilanniemi 2,
      Jyväskylä, 40100,
  Finland}
  {banajimath@gmail.com}

\author
  {Haipeng Chen}
  {School of Artificial Intelligence, Shenzhen Technology University, Shenzhen, Guangdong, China, 518118}
  {hpchen0703@foxmail.com}

\author
  {Alex Rutar}
  {
      Department of Mathematics,
      The University of British Columbia,
      1984 Mathematics Road,
      Vancouver BC V6T 1Z2,
      Canada
  }
  {alex@rutar.org}

\author
  {Wen Wang}
  {School of Mathematics and Statistics, Yunnan University, Kunming, Yunnan, China, 650504}
  {sophia\_84@126.com}

\begin{document}
\begin{abstract}
    Let $d\in\mathbb{N}$ and $\varphi\colon(0,1)\to[0,d]$.
    We prove there exists a set $F\subset\mathbb{R}^d$ whose lower spectrum $\operatorname{dim}^{\theta}_{\mathrm{L}} F$ satisfies $(1-\theta)\operatorname{dim}^{\theta}_{\mathrm{L}} F = \varphi(\theta)$ for all $\theta\in(0,1)$ if and only if for all $\lambda,\theta\in(0,1)$,
    \begin{equation*}
        \varphi(\theta) \leq \varphi(\lambda\theta) - \theta \varphi(\lambda) \leq (1-\theta) d.
    \end{equation*}
    We also obtain a similar classification result for $\underline{\operatorname{dim}}^{\theta}_{\mathrm{L}} F$.

    In contrast to the case for Assouad spectra, it is insufficient to consider homogeneous (or uniform) sets.
    Instead, we follow the approach introduced by Orgoványi--Rutar in \cite{arxiv:2510.07013} and proceed via a more general classification result for appropriate two-scale branching functions.
    \vspace{0.5cm}


    \vspace{0.2cm}

\end{abstract}

\section{Introduction}
Let $X$ be a metric space.
The goal of this paper is to study the general properties of the \emph{lower spectrum}.
For $0 < \theta < 1$, it is defined by
\begin{align*}
    \dimLs\theta X = \sup\Bigl\{s \geq 0:\exists &C>0\,\forall 0<r < 1\\*
                                                      &\inf_{x\in X}N_r(B(x,r^\theta)) \geq C \left(\frac{r^\theta}{r}\right)^s\Bigr\}.
\end{align*}
Here and elsewhere, $B(x, r)$ denotes the closed ball centred at $x$ with radius $r$ and $N_r(E)$ denotes the least number of closed balls of radius $r$ required to cover a set $E\subset X$.
The lower spectrum was motivated by the lower dimension $\dim_{\mathrm{L}} X$, which was introduced in \cite{zbl:0152.24502}.
For $0 < \theta < 1$, the inequalities
\begin{equation*}
	\dim_{\mathrm{L}} X \leq \dimLs\theta X \leq \underline{\dim}_{\mathrm{B}} X
\end{equation*}
always hold, where $\underline{\dim}_{\mathrm{B}} X$ is the lower box (or lower Minkowski) dimension.

The lower spectrum is the dual notion to the (more commonly studied) notion of the \emph{Assouad spectrum}.
We will not discuss this in any more detail outside the introduction, but for the convenience of the reader we mention the definition here: for $0<\theta < 1$, it is defined by
\begin{align*}
    \dimAs\theta X = \inf\Bigl\{s \geq 0:\exists &C>0\,\forall 0<r < 1\\*
                                                      &\sup_{x\in X}N_r(B(x,r^\theta)) \leq C \left(\frac{r^\theta}{r}\right)^s\Bigr\}.
\end{align*}
Both the Assouad spectrum and the lower spectrum were introduced in \cite{zbl:1390.28019} and calculated for various families of fractals in \cite{zbl:1407.28002}; more context can be found in the book \cite{zbl:1467.28001}.
The Assouad spectrum is useful for certain embedding problems \cite{zbl:1468.28005,zbl:1549.30098}, and plays an important role in many ``fractal'' problems in harmonic analysis \cite{zbl:08247807,zbmath:8140773,zbl:1561.42022,zbl:1526.42033}.
To the best of our knowledge, the lower spectrum has not yet appeared in such contexts.
Regardless, the motivation to study the lower spectrum is similar to the motivation for the Assouad spectrum, since they both capture a similar type of scaling information.

In \cite{zbl:1562.28062}, the author obtained a precise answer to the following question: \emph{What functions can appear as the Assouad spectrum of a subset of $\R^d$?}
A consequence of such a classification result is that it clarifies the properties of the Assouad spectrum, and it is useful for constructing examples.

The primary goal of this paper is to answer the analogous question for the lower spectrum.
A secondary goal is to collect the various properties which have appeared throughout the literature and to clarify the properties of the lower spectrum for general sets.

\subsection{Two-scale branching functions and classification}
We now state our main classification result for the lower spectrum.
\begin{itheorem}\label{it:attain}
    Let $d\in\N$ and $\varphi\colon(0,1)\to[0,d]$.
    Then there exists a set $F\subset\R^d$ such that $(1-\theta)\dimLs\theta F = \varphi(\theta)$ for all $\theta\in(0,1)$ if and only if for all $\lambda,\theta\in(0,1)$,
    \begin{equation}\label{e:ql}
        \varphi(\theta) \leq \varphi(\lambda\theta) - \theta \varphi(\lambda) \leq (1-\theta) d.
    \end{equation}
\end{itheorem}
It was already known that the first inequality of \cref{e:ql} holds for every uniformly perfect $F \subset \R^d$ (this follows from \cite[Proposition~2.1]{zbl:1434.28014}) and that the second inequality holds for every $F \subset \R^d$ (this follows from the second part of \cite[Proposition~3.1]{zbl:1467.28001}).
We will give a full proof of the necessity of \cref{e:ql} (without the uniformly perfect assumption) for completeness.
The main novelty in \cref{it:attain} is the sufficiency of \cref{e:ql}.

Since the lower spectrum is defined in a way which seems essentially identical to the Assouad spectrum (except with the various inequalities reversed), one might expect at first that the proof technique from \cite{zbl:1562.28062} could be modified in a straightforward way to work for the lower spectrum.
Perhaps surprisingly, the corresponding classification for the lower spectrum from \cref{it:attain} is qualitatively quite different.
Indeed, for the Assouad spectrum, while the reverse of the first inequality of \cref{e:ql} does appear, the second inequality is replaced with the property that $\varphi$ is $d$-Lipschitz, whereas for the lower spectrum the function $\varphi$ need not be Lipschitz on $(0,1)$.

The reason for this difference is that the lower spectrum has an essential qualitative feature that does not appear for the Assouad spectrum.
Consider the space $X = [0,1/2]\cup \{1\}$.
Then, for small $r$, $\inf_{x\in X}N_r(B(x,1)\cap X)\approx r^{-1}$, whereas $\inf_{x\in X}N_r(B(x, 1/2)\cap X) \approx 1$.
In other words, changing the radius of the larger ball by a constant factor can result in an essentially maximal change in the covering number.
This type of phenomenon does not occur with the Assouad spectrum (as long as the space $X$ is not too large, for instance, if it is doubling).
In particular, attempting to obtain a classification for the lower spectrum using uniform sets (for instance by using homogeneous Moran sets, as used in \cite{zbl:1509.28005, zbl:1562.28062}) is not possible; see \cref{it:attain-u} below.

Instead, we will approach the classification by studying a more general function, which we call the \emph{infimal two-scale branching function}.
Let $\Delta = \{(u, v) \in \R^2: v \leq u\}$.
We let $\lb = \lb_X\colon \Delta \to [0, \infty]$ denote the function
\begin{equation*}
    \lb_X(u,v) = \log \inf_{x\in X}P_{2^{-u}}(B(x, 2^{-v})).
\end{equation*}
Here and elsewhere, the logarithm is in base 2 and for a set $E$, $P_r(E)$ denotes the supremum of cardinalities of finite subsets $A\subset E$ such that the balls $B(a,2r)$ for $a\in A$ are pairwise disjoint.
For subsets $X\subset\R^d$, the function $\lb_X(u,v)$ is superadditive in an approximate sense.
It is also easy to check for all $v\in\R$ that $u\mapsto \lb_X(u,v)$ is $d$-Lipschitz in an approximate sense.
In contrast, the function $v\mapsto \lb_X(u,v)$ need not be anywhere close to Lipschitz.
This is already the case for the example $[0,1/2]\cup\{1\}$ described above.

More precisely, for $X\subset\R^d$, we will construct a function $f_X\in\mathcal{L}(d)$ (see \cref{d:L} just below) such that for all $v \leq u$,
\begin{equation}\label{e:repl-bounds}
    \lb_X(u,v) - C_d \leq f_X(u,v) \leq \lb_X(u,v-1).
\end{equation}
Here, $C_d\geq 0$ depends only on $d$; see \cref{ss:lip}.
The idea to study this branching function originates in \cite{arxiv:2510.07013}, which introduces and studies an analogous function with a supremum over covering numbers in place of the infimum over packings.

We obtain a classification result for the possible two-scale branching functions $\lb_X$.
In order to state our result, we require two definitions.
\begin{definition}\label{d:L}
    Let $\alpha \geq 0$.
    We let $\mathcal{L}(\alpha)$ denote the functions $f\colon\Delta\to [0,\infty)$ such that the following hold:
    \begin{enumerate}[nl,r]
        \item\label{i:2} For all $v \leq w \leq u$,
            \begin{equation}\label{e:mono}
                f(u,v) \geq f(u,w) + f(w,v).
            \end{equation}
        \item\label{i:3} For all $v$, the function $u\mapsto f(u,v)$ is $\alpha$-Lipschitz.
    \end{enumerate}
\end{definition}
It follows from \cref{e:mono} that $f(u,u) = 0$ for all $u\in\R$ and that $f(u,v)$ is an increasing function of $u$ and a decreasing function of $v$.
Note that condition \cref{i:2} implies the first inequality in \cref{e:ql}, and \cref{i:3} implies the second inequality in \cref{e:ql}; see \cref{l:L-in-H}.

We also require an equivalence relation on the functions $f\colon\Delta\to[0,\infty)$.
\begin{definition}\label{d:csim}
    Let $f,g\colon\Delta\to[0,\infty)$.
    We write $f\sim g$ if there exists a constant $z \geq 0$ such that for all $0 \leq v \leq u$,
    \begin{itemize}[nl]
        \item $f(u, v+z) -z \leq g(u,v)$ whenever $v \leq u-z$, and
        \item $g(u,v) \leq f(u, v-z) + z$ whenever $v \geq z$.
    \end{itemize}
\end{definition}
One can think of the relation $f\sim g$ as stating that $f$ and $g$ are equal up to a fixed additive error term; the actual statement allows some flexibility since we have no control on the function $v\mapsto\lb_X(u,v)$ at specific points in general (other than knowing that it is monotone).
Below, we use the subscript $f\sim_d g$ to indicate that the implicit constants depend only on $d$.
\begin{itheorem}\label{it:two-scale}
    Let $d\in\N$.
    \begin{enumerate}[nl,r]
        \item\label{i:fL} If $X\subset\R^d$, then there is a function $f_X\in\mathcal{L}(d)$ such that $f_X\sim_d\lb_X$.
        \item\label{i:ext} If $f\in\mathcal{L}(d)$, then there exists a set $X\subset\R^d$ such that $f\sim_d\lb_X$.
    \end{enumerate}
\end{itheorem}

Let us give a brief outline of the proofs of \cref{it:attain} and \cref{it:two-scale}.
In \cref{ss:twoscale}, we establish basic properties of the two-scale branching function $\lb_X$.
Most of the work consists of manipulation to establish the Lipschitz property and superadditivity; this is done in \cref{ss:lip}.
By the end of \cref{ss:lip}, we will have proven \cref{it:two-scale}~\cref{i:fL}.
Then in \cref{ss:attain-ts}, we prove \cref{it:two-scale}~\cref{i:ext}.
(We construct the set as a countable disjoint union of the sets described in \cref{ss:unif-ct}.)

In \cref{ss:norm}, we introduce a normalized limit of branching functions and prove in \cref{l:L-in-H} that the set of limits is precisely those functions which satisfy the inequalities in \cref{it:attain}.
We then note in \cref{ss:cont-lip} that the normalized limit is continuous by establishing certain weak Lipschitz properties.
Finally, in \cref{ss:equiv-l}, we show using continuity that the limit is unchanged by the equivalence relation in \cref{d:csim}, and then prove that it corresponds exactly to the lower spectrum of $X$.
The proof of \cref{it:attain} then follows by \cref{it:two-scale}.

\subsection{Additional classification results}\label{ss:add}
Next, we state two additional classification results.

The first classification result concerns uniform sets.
Heuristically, a uniform set is one for which the function $x\mapsto N_r(B(x,R))$ is approximately constant on $X$, independently of $0<r \leq R$.
In practice, it is only necessary that the error term is sufficiently small as a function of $r$; see \cref{d:unif} for the precise definition.

In \cite{zbl:1562.28062}, it is proven that every Assouad spectrum function can be obtained as the Assouad spectrum of a uniform set (see also \cite[§2.4]{zbmath:08078927}).
In contrast, for the lower spectrum, we saw in the previous section that the essential difficulty involved non-uniform sets.
This motivated us to pursue the classification using a more general classification result for two-scale branching functions.

To give a rigorous justification that it would not suffice to use uniform sets in the proof of \cref{it:attain}, we prove the following result in \cref{s:unif}.
\begin{itheorem}\label{it:attain-u}
    Let $d\in\N$ and $\varphi\colon(0,1)\to [0,d]$.
    Then there is a uniform set $F\subset\R^d$ with $(1-\theta)\dimLs\theta F = \varphi(\theta)$ for all $\theta\in(0,1)$ if and only if $\varphi$ is $d$-Lipschitz and
    \begin{equation*}
        \varphi(\theta) \leq \varphi(\lambda\theta) - \theta \varphi(\lambda)
    \end{equation*}
    for all $\lambda,\theta\in(0,1)$.
\end{itheorem}
Comparing this to \cref{it:attain}, the key difference is that the second inequality in \cref{e:ql} is replaced with the $d$-Lipschitz property.
This results in a strictly smaller family of functions; see \cref{l:U-inc}.

The proof of \cref{it:attain-u} is very similar to the proof of \cite[Theorem~A]{zbl:1562.28062}.
Here, we write the proof in the language of branching functions, which makes it quite a bit easier to understand the algebra.
The details can be found in \cref{ss:unif-attain}.

Our final classification result concerns a monotone version of the lower spectrum.
For $0 < \theta < 1$, the \emph{monotone lower spectrum} is given by
\begin{align*}
    \dimuLs\theta X = \sup\Bigl\{s \geq 0:\exists &C>0\,\forall 0<r \leq r^\theta \leq R < 1\\*
                                                       &\inf_{x\in X}N_r(B(x,R)\cap X) \geq C \left(\frac{R}{r}\right)^s\Bigr\}.
\end{align*}
This was first defined in \cite{zbl:1371.28003}, where it was called the quasi-lower spectrum.
For $X\subset\R^d$, it was proven in \cite[Theorem~A.2]{zbl:1479.28010} that for $0<\theta<1$,
\begin{equation}\label{e:vari}
    \dimuLs\theta X = \inf_{0<\lambda\leq \theta}\dimLs\lambda X.
\end{equation}
See also \cite[Theorem~1.1]{zbl:1434.28014}.
In \cref{s:monot}, we also give a short and simple proof of \cref{e:vari} when $X$ is an arbitrary metric space.

Combining \cref{it:attain} and \cref{e:vari} with short arguments provided in \cref{ss:mono}, we obtain the following classification result for the monotone lower spectrum.
\begin{icorollary}\label{ic:mono-class}
    Let $d\in\N$ and $\varphi\colon(0,1)\to[0,d]$.
    Then there exists $F\subset\R^d$ such that $(1-\theta)\dimuLs\theta F = \varphi(\theta)$ for all $\theta\in(0,1)$ if and only if $\theta\mapsto \varphi(\theta)/(1-\theta)$ is decreasing and
    \begin{equation*}
        \varphi(\lambda\theta) - \theta \varphi(\lambda) \leq (1-\theta) d
    \end{equation*}
    for all $\lambda,\theta\in(0,1)$.
\end{icorollary}
We also obtain a complete description of the set of functions $\theta\mapsto(1-\theta)\dimuLs\theta F$ for subsets $F\subset\R^d$ as infima of an explicit family of functions; see \cref{t:mono-class} for more detail.
As was the case for uniform sets, the family of functions in \cref{ic:mono-class} is strictly smaller since the lower spectrum need not be decreasing in general.
Such an example was first constructed in \cite[§8]{zbl:1390.28019}; we provide another example with an explicit formula for the lower spectrum in \cref{ss:non-mono}.
In \cref{p:extra} we also mention (without proof) an analogous result for the possible monotone lower spectra of uniform sets.

Finally, in \cref{s:ineq} we provide a self-contained discussion of the inequalities appearing in \cref{it:attain}, \cref{it:attain-u}, and \cref{ic:mono-class}, with a particular focus on geometric interpretation of the bounds.

\subsection{Further research}
To conclude the introduction, we suggest some directions for future research.

First, we pose a question concerning the simultaneous classification of the Assouad and lower spectra.
Let $d\in\N$ and $\varphi_{\mathrm{A}}\colon(0,1)\to[0,d]$.
We recall from \cite[Theorem~A]{zbl:1562.28062} that there exists a set $F\subset\R^d$ such that $(1-\theta)\dimAs\theta F = \varphi_{\mathrm{A}}(\theta)$ for all $\theta\in(0,1)$ if and only if $\lim_{\theta\nearrow1}\varphi_{\mathrm{A}}(\theta)=0$, $\varphi_{\mathrm{A}}$ is $d$-Lipschitz, and for all $\lambda,\theta\in(0,1)$,
\begin{equation}\label{e:aq}
    \varphi_{\mathrm{A}}(\theta) + \theta\varphi_{\mathrm{A}}(\lambda) \geq \varphi_{\mathrm{A}}(\lambda\theta) \geq \varphi_{\mathrm{A}}(\theta).
\end{equation}
However, it turns out that if we know information about $\dimAs\theta F$ or $\dimLs\theta F$, then we can improve \cref{e:ql} and \cref{e:aq}.
It is not too difficult to show, for $\varphi_{\mathrm{L}}(\theta) \coloneqq (1-\theta)\dimLs\theta F$ and $\varphi_{\mathrm{A}}(\theta) \coloneqq (1-\theta)\dimAs\theta F$, that the following stronger inequalities hold for all $\lambda,\theta\in(0,1)$:
\begin{equation}\label{e:joint}
    \begin{gathered}
        \varphi_{\mathrm{L}}(\theta) \leq \varphi_{\mathrm{L}}(\lambda\theta) - \theta \varphi_{\mathrm{L}}(\lambda) \leq \varphi_{\mathrm{A}}(\theta),\\*
        \varphi_{\mathrm{L}}(\theta) \leq \varphi_{\mathrm{A}}(\lambda\theta) - \theta \varphi_{\mathrm{A}}(\lambda) \leq \varphi_{\mathrm{A}}(\theta).
    \end{gathered}
\end{equation}
For instance, the upper bound of the first equation in \cref{e:joint} is precisely the upper bound in the second part of \cite[Proposition~3.1]{zbl:1467.28001}.
Together with the monotonicity of $\varphi_{\mathrm{A}}$, we recover the original inequalities from \cref{e:joint} since $\varphi_{\mathrm{L}}(\theta) \geq 0$ and $\varphi_{\mathrm{A}}(\theta) \leq d(1-\theta)$ for all $\theta\in(0,1)$.
One might ask if there are no additional joint inequalities.
\begin{question}
    Let $d\in\N$ and suppose $\varphi_{\mathrm{L}},\varphi_{\mathrm{A}}\colon(0,1)\to[0,d]$ are such that $\varphi_{\mathrm{A}}$ is decreasing and $d$-Lipschitz, $\lim_{\theta\nearrow1}\varphi_{\mathrm{A}}(\theta)=0$, and $\varphi_{\mathrm{L}},\varphi_{\mathrm{A}}$ satisfy \cref{e:joint}.
    Does there exist a set $F\subset\R^d$ such that $(1-\theta)\dimLs\theta F = \varphi_{\mathrm{L}}(\theta)$ and $(1-\theta)\dimAs\theta F = \varphi_{\mathrm{A}}(\theta)$ for all $\theta\in(0,1)$?
\end{question}

One could also ask about the \emph{modified} versions of $\dimLs\theta$ and $\dimuLs\theta$, defined by
\begin{equation*}
	\dimMLs\theta X \coloneqq \sup\{\dimLs\theta Y : Y \subset X \}; \qquad \dimuMLs\theta X \coloneqq \sup\{\dimuLs\theta Y : Y \subset X \}.
\end{equation*}
The motivation for these definitions is that, unlike $\dimLs\theta$ and $\dimuLs\theta$, the modified versions are monotone for subsets: if $Y \subset X$, then $\dim^{\theta}_{\mathrm{ML}} Y \leq \dim^{\theta}_{\mathrm{ML}} X$ and $\dimuMLs\theta Y \leq \dimuMLs\theta X$.
It is easy to see that if $X$ is uniform, then $\dimMLs\theta X = \dimLs\theta X$ and $\dimuMLs\theta X = \dimuLs\theta X$.
In particular, \cref{it:attain-u} shows that many functions can appear as the modified lower spectrum of subsets of $\R^d$.
However, the proof of \cref{it:attain} does not work for the modified lower spectrum since the sets $E_k$ in the construction may have larger lower spectrum than the resulting set $X$.
This raises the following question.
\begin{question}
    Let $d\in\N$.
    What is the set of functions $\theta\mapsto\dimMLs\theta F$ for $F\subset\R^d$?
    What about $\theta\mapsto\dimuMLs\theta F$?
\end{question}

\section{Two-scale branching functions}\label{s:ts}
Throughout this section, $X$ is a non-empty metric space.
We formally introduce the infimal two-scale branching function and prove \cref{it:two-scale}.

\subsection{An infimal two-scale branching function}\label{ss:twoscale}
We begin with well-separated packings, which we find to be more convenient to work with than covers.
\begin{definition}
    For $E\subset X$ and $r > 0$, we let $P_r(E)$ denote the supremum over $k\in\N$ for which there exist points $\{x_1,\ldots,x_k\}\subset E$ such that $B(x_i, 2r) \cap B(x_j, 2r) = \varnothing$ for all $i \neq j$.
    We call such a collection $\{x_1,\ldots,x_k\}$ a \emph{(strong) $r$-packing} of $E$.
\end{definition}
Clearly if $E$ is non-empty, then $P_r(E)\geq 1$.
Up to a change of radius by a constant factor, packings are equivalent to covers.
For $E\subset X$, let $N_r(E)$ denote the least number of closed balls of radius $r$ required to cover $E$, in which case it is easy to check that
\begin{equation}\label{e:cov-eq}
    N_{4r}(E) \leq P_r(E) \leq N_r(E).
\end{equation}
We will see in a moment why we find it more convenient to use packings with additional separation instead of covers, or even regular packings.

We now introduce an appropriate two-scale branching function, which we call the \emph{(infimal) two-scale branching function}.
Let $\Delta = \{(u, v) \in \R^2: v \leq u\}$.
We let $\lb = \lb_X\colon \Delta \to [0, \infty]$ denote the function
\begin{equation*}
    \lb_X(u,v) = \log \inf_{x\in X}P_{2^{-u}}(B(x, 2^{-v})).
\end{equation*}
Here and elsewhere, $\log$ denotes the logarithm in base $2$.
It is immediate from the definition that $\lb$ is non-negative, increasing in $u$, decreasing in $v$, and satisfies $\lb(u,u) = 0$.

\subsection{Lipschitz and superadditivity properties}\label{ss:lip}
Let us begin with the superadditivity property of $\lb_X$.
Unfortunately, $\lb_X$ is only superadditive in an approximate sense, so we first explain how to modify it to be genuinely superadditive.

For a non-negative function $h$ on $\Delta$ with $h(u,u) = 0$, define the \emph{superadditive hull}
\begin{equation}\label{e:hull}
    \mathcal{S}h(u,v) = \sup_{v=t_0\leq t_1\leq \cdots\leq t_n=u}\sum_{i=1}^n h(t_i,t_{i-1})
\end{equation}
for $v \leq u$.
It is easy to see that $\mathcal{S}h$ is the minimal superadditive function bounded below by $h$: for $v \leq w \leq u$, $\mathcal{S}h(u,v) \geq \mathcal{S}h(u, w) + \mathcal{S}h(w, v)$.
\begin{lemma}\label{l:hull}
    Let $X$ be a non-empty metric space.
    Then for all $v \leq u$,
    \begin{equation*}
        \lb_X(u,v) \leq \mathcal{S}\lb_X(u,v) \leq \lb_X(u,v-1).
    \end{equation*}
\end{lemma}
\begin{proof}
    The first inequality is immediate from the definition.
    For the second, let $v \leq u$ be arbitrary.
    If $\mathcal{S}\lb_X(u,v) = 0$ there is nothing to prove.
    Otherwise, let $v = t_0 \leq t_1 \leq\cdots\leq t_n = u$ be some partition.
    By merging intervals in the partition for which $\lb_X(t_i, t_{i-1}) = 0$ and using monotonicity, we may assume that $\lb_X(t_i, t_{i-1}) > 0$ for all $1\leq i\leq n$.
    Moreover, if $2r \geq R$, clearly $P_r(B(x,R)) = 1$.
    In particular, writing $r_i = 2^{-t_i}$, we have that $r_i < r_{i-1}/2$ for all $1 \leq i \leq n$.

    Now let $x\in X$ be arbitrary and consider the ball $B(x, r_0)$.
    Begin with the strong $r_0$-packing $\mathcal{P}_0 = \{x\}$: we inductively construct a sequence of strong $r_i$-packings.
    Suppose we have constructed a strong $r_{i-1}$-packing $\mathcal{P}_{i-1}$.
    For each $z\in\mathcal{P}_{i-1}$, choose a strong $r_i$-packing of the ball $B(z, r_{i-1})$, and let $\mathcal{P}_i$ denote the union over these packings.
    To see that $\mathcal{P}_i$ is indeed a strong $r_i$-packing, observe that if $y\in B(z, r_{i-1})$, since $r_i < r_{i-1}/2$,
    \begin{equation*}
        B(y,2r_i)\subset B(z,r_{i-1}+2r_i)\subset B(z,2r_{i-1}).
    \end{equation*}
    The final packing $\mathcal{P}_n$ is a strong $r_n$-packing of $B(x, 2r_0)$ since $r_0 + \cdots + r_{n-1} < 2r_0$.
    Taking a supremum over all inductive choices of $r_i$-packings,
    \begin{equation*}
        P_{r_n}(B(x, 2 r_0)) \geq \prod_{i=1}^n\inf_{y\in X}P_{r_i}(B(y,r_{i-1})).
    \end{equation*}
    Since $x\in X$ was arbitrary, it follows that $\lb_X(u, v-1)\geq \sum_{i=1}^n \lb_X(t_i,t_{i-1})$ as required.
\end{proof}

Now we turn our attention to the Lipschitz property.
Unlike the two-scale branching function defined with a supremum, the function $\lb_X$ need not be Lipschitz in both parameters.
We recall the example from the introduction: consider the case that $X = [0,1/2]\cup\{1\}$ and let $u$ be large.
For $1 < v \leq u$, $B(1, 2^{-v}) = \{1\}$ so that $\lb_X(u, v) = 0$.
On the other hand, if $v \leq 0$, then $\lb_X(u, v) = u + O(1)$.
Therefore, we see that $\lb_X$ can be highly non-Lipschitz as a function of $v$.

Regardless, $\lb_X$ must be Lipschitz in an approximate sense as a function of $u$.
\begin{lemma}\label{l:approx-br}
    Suppose $X\subset\R^d$.
    Then for all $(u,v)\in\Delta$ and $z \geq 0$,
    \begin{equation*}
        0 \leq \lb_X(u + z, v) - \lb_X(u, v) \leq dz + O_d(1).
    \end{equation*}
\end{lemma}
\begin{proof}
    The first inequality is just monotonicity in $u$.
    The other inequality follows from \cref{e:cov-eq}.
    Consider a fixed ball $B(x, 2^{-v})$.
    Then
    \begin{align*}
        P_{2^{-(u+z)}}(B(x, 2^{-v})) &\leq N_{2^{-(u+z)}}(B(x, 2^{-v}))\\
                                     &\lesssim_d 2^{dz} N_{2^{-u}}(B(x, 2^{-v}))\\
                                     &\lesssim_d 2^{dz} N_{4 \cdot 2^{-u}}(B(x, 2^{-v}))\\
                                     &\leq 2^{dz}P_{2^{-u}}(B(x, 2^{-v})).
    \end{align*}
    Taking an infimum over $x\in X$ and then a logarithm yields the claimed inequality.
\end{proof}
Now, recall the definition of $\mathcal{L}(\alpha)$ from \cref{d:L}.
We also let $\mathcal{L}$ denote those functions satisfying \cref{d:L}~\cref{i:2}.

By \cref{l:hull}, $\mathcal{S}\lb_X\in\mathcal{L}$.
It remains to show that $\mathcal{S}\lb_X$ can be replaced with a function which is genuinely in $\mathcal{L}(d)$, without changing the function too much.
\begin{proposition}\label{p:repl}
    Let $\alpha\geq 0$ and let $h\colon\Delta\to[0,\infty)$ satisfy $h(u,u) = 0$ for all $u\in\R$.
    Suppose $\eta\colon\R\to[0,\infty)$ is such that for all $v\leq w\leq u$,
    \begin{equation*}
        h(u,v) - h(w,v) \leq \alpha(u-w) + \eta(u).
    \end{equation*}
    Then there exists a function $g\in\mathcal{L}(\alpha)$ such that
    \begin{equation*}
        h(u,v) - \eta(u) \leq g(u,v) \leq \mathcal{S}h(u,v)
    \end{equation*}
    for all $(u,v)\in\Delta$.
\end{proposition}
\begin{proof}
    Define
    \begin{equation*}
        a(u,v) = \inf_{v\leq t\leq u}\bigl(h(t,v) + \alpha(u-t)\bigr).
    \end{equation*}
    By assumption and the choice $t=u$,
    \begin{equation*}
        h(u,v) - \eta(u) \leq a(u,v) \leq h(u,v).
    \end{equation*}
    Also, $a$ is non-negative, $a(u,u) = 0$, and by considering the infimum only over $t\leq w$,
    \begin{equation}\label{e:repl-lip}
        a(u,v) \leq a(w,v) + \alpha(u-w)
    \end{equation}
    for all $v\leq w\leq u$.
    In particular, taking $w=v$ gives $a(u,v)\leq\alpha(u-v)$, so $g = \mathcal{S}a$ is finite, non-negative, and superadditive.

    To check the Lipschitz property, fix $v\leq w\leq u$ and a partition $v=t_0\leq\cdots\leq t_n=u$.
    Let $j$ be such that $t_{j-1}\leq w\leq t_j$.
    By \cref{e:repl-lip}, $a(t_j,t_{j-1})\leq a(w,t_{j-1})+\alpha(t_j-w)$.
    Combining this with the inequality $a(u,v)\leq\alpha(u-v)$ on the remaining intervals in the partition gives
    \begin{align*}
        \sum_{i=1}^n a(t_i,t_{i-1})
        &\leq \sum_{i=1}^{j-1} a(t_i,t_{i-1}) + a(w,t_{j-1}) + \alpha(u-w)\\
        &\leq g(w,v) + \alpha(u-w).
    \end{align*}
    Taking the supremum over all partitions gives $g(u,v)-g(w,v)\leq\alpha(u-w)$, and the reverse bound $g(u,v)-g(w,v)\geq 0$ follows from non-negativity and superadditivity.
    Thus $g\in\mathcal{L}(\alpha)$.
    The proof is complete since $g\geq a$ and $g=\mathcal{S}a\leq\mathcal{S}h$.
\end{proof}
We can now complete the proof of the first half of \cref{it:two-scale}.
\begin{restatementbase}[of~\cref{it:two-scale}~\cref{i:fL}]
    Let $X\subset\R^d$ be non-empty.
    Then there exists an $f_X\in\mathcal{L}(d)$ such that $f_X \sim_d \lb_X$.
\end{restatementbase}
\begin{proof}
    Let $C_d\geq 0$ be the constant from \cref{l:approx-br}.
    Applying \cref{p:repl} (with $h=\lb_X$, $\alpha=d$, and $\eta(u)=C_d$) followed by \cref{l:hull}, we obtain a function $f_X\in\mathcal{L}(d)$ such that
    \begin{equation*}
        \lb_X(u,v) - C_d \leq f_X(u,v) \leq \mathcal{S}\lb_X(u,v) \leq \lb_X(u,v-1).
    \end{equation*}
    This proves \cref{e:repl-bounds}, and in particular $f_X\sim_d\lb_X$.
\end{proof}
\subsection{Attainable two-scale branching functions}\label{ss:attain-ts}
Now, we proceed with the proof of the second half of \cref{it:two-scale}.
Recall the equivalence relation introduced in \cref{d:csim} in the introduction.

We also introduce a definition highlighting the precise properties that we will require for an arrangement of sets in our construction.
\begin{definition}
    A \emph{geometric sequence} of compact sets is a sequence of sets $\{E_k\}_{k=0}^\infty$ for which there is a constant $m \geq 0$ and points $x_k\in E_k$ so that, for all integers $k \geq 0$,
    \begin{enumerate}[nl,r]
        \item $E_k \subset B(x_k, 2^{-k})$,
        \item $E_j \cap B(x_k, 2^{-k}) = \varnothing$ for all $j \neq k$,
        \item $E_k \subset B(0, 2^{-k + m})$.
    \end{enumerate}
\end{definition}
Given non-empty compact sets $\{F_k\}_{k=0}^\infty$ with $\diam F_k \leq 2^{-k}$, it is easy to see that there are translations $t_k\in\R^d$ so that $\{F_k + t_k\}_{k=0}^\infty$ is a geometric sequence.
\begin{figure}[t]
    \centering
    \begin{tikzpicture}[xscale = 5, bar/.style={line width=1pt, line cap=butt, shorten >=1pt}]
  \def\rs{0.4}
  \def\gs{0.5}
  \def\igap{0}

  \pgfmathsetmacro{\Etwo}{4*\rs+\gs}
  \pgfmathsetmacro{\Ethree}{8*\rs+2*\gs}

  \draw[bar] (0, 0) -- node[above]{$E_0$} (1, 0);

  \draw[bar] (0, {-\rs}) -- (0.5, {-\rs});
  \draw[bar] ({0.5}, {-\rs}) -- (1, {-\rs});

  \draw[bar] (0, {-2*\rs}) -- (0.25, {-2*\rs});
  \draw[bar] (0.5, {-2*\rs}) -- (0.75, {-2*\rs});

  \draw[bar] (0, {-3*\rs}) -- node[below]{$\vdots$} (0.125, {-3*\rs});
  \draw[bar] (0.125, {-3*\rs}) -- node[below]{$\vdots$} (0.25, {-3*\rs});
  \draw[bar] (0.5, {-3*\rs}) -- node[below]{$\vdots$} (0.625, {-3*\rs});
  \draw[bar] (0.625, {-3*\rs}) -- node[below]{$\vdots$} (0.75, {-3*\rs});

  \begin{scope}[xshift={1.5cm}]
    \draw[bar] (0, 0) -- node[above]{$E_1$} (0.5, 0);

    \draw[bar] (0.0, {-1*\rs}) -- (0.25, {-1*\rs});

    \draw[bar] (0.0, {-2*\rs}) -- (0.125, {-2*\rs});
    \draw[bar] (0.125, {-2*\rs}) -- (0.25, {-2*\rs});

    \draw[bar] (0.0, {-3*\rs}) -- node[below]{$\vdots$} ({1/16}, {-3*\rs});
    \draw[bar] ({1/8}, {-3*\rs}) -- node[below]{$\vdots$} ({1/8+1/16}, {-3*\rs});
  \end{scope}

  \begin{scope}[xshift={2.25cm}]
      \draw[bar] (0, 0) -- node[above]{$E_2$} (0.25, 0);

    \draw[bar] (0.0, {-1*\rs}) -- (0.125, {-1*\rs});
    \draw[bar] (0.125, {-1*\rs}) -- (0.25, {-1*\rs});

    \draw[bar] (0.0, {-2*\rs}) -- ({1/16}, {-2*\rs});
    \draw[bar] ({1/16}, {-2*\rs}) -- ({2/16}, {-2*\rs});
    \draw[bar] ({2/16}, {-2*\rs}) -- ({3/16}, {-2*\rs});
    \draw[bar] ({3/16}, {-2*\rs}) -- ({4/16}, {-2*\rs});

    \draw[bar] (0.0, {-3*\rs}) -- node[below]{$\vdots$} ({1/32}, {-3*\rs});
    \draw[bar] ({1/16}, {-3*\rs}) -- node[below]{$\vdots$} ({1/16 + 1/32}, {-3*\rs});
    \draw[bar] ({2/16}, {-3*\rs}) -- node[below]{$\vdots$} ({2/16 + 1/32}, {-3*\rs});
    \draw[bar] ({3/16}, {-3*\rs}) -- node[below]{$\vdots$} ({3/16 + 1/32}, {-3*\rs});
  \end{scope}

  \node at (2.6, 0) {$\hdots$};
\end{tikzpicture}
    \caption{A depiction of the set constructed in the proof of \cref{it:two-scale}~\cref{i:ext}.
        Each set $E_k$ is constructed as an intersection of nested families of dyadic cubes, following the procedure in \cref{ss:unif-ct}, with the subdivision sequence of each $E_k$ depending on $f_k$.
        The spacing between $E_k$ and $E_{k+1}$ is proportional to $2^{-k}$.
        The sets $E_k$ for $k \geq 3$ are omitted.
    }
    \label{f:geom}
\end{figure}
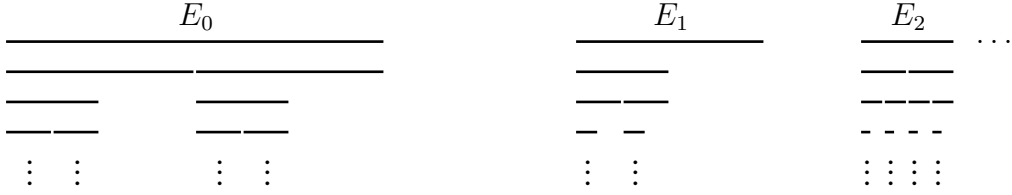
\begin{restatementbase}[of~\cref{it:two-scale}~\cref{i:ext}]
    Let $d\in\N$ and $f\in\mathcal{L}(d)$ be arbitrary.
    Then there exists a non-empty compact set $X\subset\R^d$ such that $f \sim_d \lb_X$.
\end{restatementbase}
\begin{proof}
    First, for each integer $k \geq 0$, let $f_k$ denote the function $f_k(u) = f(u,k)$ for $u \geq k$, and $f_k(u) = 0$ otherwise.
    Since $f\in\mathcal{L}(d)$, $f_k$ is an increasing $d$-Lipschitz function.
    Therefore, by \cref{p:unif-ct} combined with scaling by $d^{-1/2}$ and an appropriate translation, there exists a set $E_k\subset B(x_k, 2^{-k})$ with branching function
    \begin{equation*}
        \lb_{E_k}(u,v) = f_k(u) - f_k(v) + O(1)
    \end{equation*}
    such that $\{E_k\}_{k=0}^\infty$ is a geometric sequence with respect to the points $x_k$ and constant $m\geq 0$ depending only on $d$.
    Let
    \begin{equation*}
    		X \coloneqq \{0\}\cup \bigcup_{k=0}^\infty E_k.
    	\end{equation*}

    The upper bound is immediate: for each $k\geq 0$,
    \begin{equation*}
        \lb_{X}(u,k) \leq \log P_{2^{-u}}(X\cap B(x_k, 2^{-k})) = f_k(u) - f_k(k) + O(1) = f(u, k) + O(1).
    \end{equation*}
    For the lower bound, let $x\in X$ be arbitrary.
    First suppose either $x = 0$ or $x\in E_n$ for some $n \geq k$.
    Since $x\in B(0, 2^{-k+m})$ and $E_k\subset B(0, 2^{-k+m})$, it follows that $E_k \subset B(x, 2^{-k + m + 1})$ and therefore
    \begin{align*}
        \log P_{2^{-u}}(X\cap B(x, 2^{-k + m + 1})) &\geq \log P_{2^{-u}}(E_k\cap B(x, 2^{-k + m + 1}))\\
                                                    &= f_k(u) - f_k(k - m - 1) + O(1)\\
                                                    &= f(u, k) + O(1).
    \end{align*}
    Otherwise, $x\in E_n$ for some $0 \leq n\leq k$.
    Then using superadditivity of $f$,
    \begin{align*}
        \log P_{2^{-u}}(X\cap B(x, 2^{-k + m + 1})) &\geq \log P_{2^{-u}}(E_n\cap B(x, 2^{-k + m + 1}))\\
                                                    & = f_n(u) - f_n(k - m - 1) + O(1)\\
                                                    &= f(u, n) - f(k, n) + O(1)\\
                                                    &\geq f(u, k) + O(1).
    \end{align*}
    Since this handles all possible cases of $x\in X$, we have shown that
    \begin{equation*}
        \lb_X(u, k - m - 1) \geq f(u, k) + O(1)
    \end{equation*}
    for all integers $k \geq 0$.
    Combined with the upper bound, it follows that $\lb_X \sim f$, as required.
\end{proof}
\begin{remark}
    In the above proof, we construct a compact set since we are only interested in small scales (that is, the values of $\lb_X(u,v)$ for $0 \leq v \leq u$).
    If one is also interested in large scales, one can instead take a union $\bigcup_{k\in\Z} E_k$ where the sets $E_k$ for $k < 0$ are constructed starting at a dyadic cube of side-length $2^{-k}$ which is much larger than $1$.
    The resulting set will be unbounded, but the proof is the same.
\end{remark}

\section{Classification for the lower spectrum}\label{s:lower}
In the previous section, we proved the classification for the two-scale branching functions $\lb_X$.
In this section, we relate the function $\lb_X$ to the lower spectrum and prove \cref{it:attain} and \cref{ic:mono-class}.

\subsection{Domain restriction}
Since the lower spectrum only depends on the covering properties at small scales, it is convenient to restrict the domain $\Delta$.
\begin{definition}
    Let $a\in\R$ and $\alpha\geq 0$.
    We write $\Delta_a = \{(u,v)\in\R^2: a\leq v\leq u\}\subset\Delta$.
    Similarly, let $\mathcal{L}_a(\alpha)$ denote the functions $f\colon\Delta_a\to[0,\infty)$ satisfying the conditions of \cref{d:L} on $\Delta_a$.
\end{definition}
Clearly, functions in $\mathcal{L}(\alpha)$ are naturally functions in $\mathcal{L}_a(\alpha)$.
It is also easy to check that one can extend functions in $\mathcal{L}_a(\alpha)$ to $\mathcal{L}(\alpha)$.
\begin{lemma}\label{l:domain-ext}
    Let $a\in\R$ and $\alpha\geq 0$.
    For $f\in\mathcal{L}_a(\alpha)$, define $\lex f\colon\Delta\to[0,\infty)$ by
    \begin{equation*}
        \lex f(u,v) = \begin{cases}
            f(u,v) &: v\geq a,\\*
            \alpha(u-v) &: v<a.
        \end{cases}
    \end{equation*}
    Then $\lex f$ is the pointwise maximal element of $\mathcal{L}(\alpha)$ which coincides with $f$ on $\Delta_a$.
\end{lemma}
\begin{proof}
    It is easy to check that $g(u,v) \leq \alpha(u-v)$ for all $g\in\mathcal{L}(\alpha)$, and similarly for $g\in\mathcal{L}_a(\alpha)$.
    In particular, it suffices to check that $\lex f\in\mathcal{L}(\alpha)$.

    That $\lex f$ satisfies \cref{d:L}~\cref{i:3} is immediate from the definition.
    For superadditivity, let $v\leq w\leq u$.
    If $v\geq a$, this is just superadditivity of $f$.
    If $v<a$, then
    \begin{equation*}
        \lex f(u,w)+\lex f(w,v)\leq\alpha(u-w)+\alpha(w-v)=\lex f(u,v)
    \end{equation*}
    as required.
\end{proof}

\subsection{A normalized limit of branching functions}\label{ss:norm}
Let $f\colon \Delta \to [0,\infty]$ be arbitrary.
We let $\Lambda(f)$ denote the function defined for $\theta\in (0,1)$ by the rule
\begin{equation*}
    \Lambda(f)(\theta) = \liminf_{u\to\infty}\frac{f(u, \theta u)}{u}.
\end{equation*}
Note since $\theta > 0$ that $\Lambda$ only depends on the values of $f$ on $\Delta_a$ for any $a\in\R$.
As is the case in \cref{r:cont-ext}, we will extend $\Lambda(f)$ continuously to $[0,1]$ when this is well-defined.
Note, however, that even when this is possible, $\Lambda(f)(0)$ need not equal $\liminf_{u\to\infty}f(u,0)/u$.

We will see that the range of $\Lambda$ restricted to $\mathcal{L}(\alpha)$ is the following space of functions.
\begin{definition}\label{d:H}
    Let $\alpha \geq 0$.
    We let $\mathcal{H}(\alpha)$ denote the functions $\varphi\colon (0,1)\to [0,\alpha]$ such that the following hold:
    \begin{enumerate}[nl,a]
        \labelitem{i:superadd}{S} \emph{Superadditive:} For all $\lambda,\theta\in(0,1)$,
            \begin{equation*}
                \varphi(\lambda\theta) \geq \varphi(\theta) + \theta\varphi(\lambda).
            \end{equation*}
        \labelitem{i:conv}{W} \emph{Weak Lipschitz:} For all $\lambda,\theta\in(0,1)$,
            \begin{equation*}
                \varphi(\lambda\theta) \leq (1-\theta) \alpha + \theta\varphi(\lambda).
            \end{equation*}
    \end{enumerate}
\end{definition}
Let us make some simple observations about the functions in $\mathcal{H}(\alpha)$.
\begin{enumerate}[nl]
    \item It follows from \eqref{i:superadd} that $\varphi$ is decreasing.
        Since $\varphi$ is bounded, we may define $\varphi(0) \coloneqq \lim_{\theta \to 0}\varphi(\theta)$.
    \item\label{i:iter} Iterating \eqref{i:superadd}, for $0<\theta<1$ and $n\geq1$,
        \begin{equation*}
            \alpha \geq \varphi(\theta^n) \geq \frac{1-\theta^n}{1-\theta}\varphi(\theta).
        \end{equation*}
        Taking $n\to\infty$ gives $\varphi(\theta)\leq\alpha(1-\theta)$.
        Hence $\lim_{\theta\nearrow1}\varphi(\theta)=0$, and so we may define $\varphi(1)=0$.
    \item With the above endpoint values, it is easy to check that \eqref{i:superadd} and \eqref{i:conv} hold for all $\lambda,\theta\in[0,1]$.
    \item One can interpret \eqref{i:conv} as a weak type of convexity at $0$: for all $\lambda\in[0,1]$, the graph of $\varphi$ on $[0,\lambda]$ lies below the line segment connecting $(0,\alpha)$ and $(\lambda, \varphi(\lambda))$.
    \item Using the above observation, it follows that \eqref{i:conv} holds if and only if the function
        \begin{equation*}
            \theta\mapsto \frac{\alpha - \varphi(\theta)}{\theta}
        \end{equation*}
        is decreasing.
\end{enumerate}
To motivate the space $\mathcal{H}(\alpha)$, we observe that it is precisely the class of normalized limits of elements of $\mathcal{L}(\alpha)$.
\begin{proposition}\label{l:L-in-H}
    Let $\alpha\geq 0$.
    Then $\Lambda(\mathcal{L}(\alpha)) = \mathcal{H}(\alpha)$.
\end{proposition}
\begin{proof}
    First let $f\in\mathcal{L}(\alpha)$.
    We already saw earlier that $\varphi = \Lambda(f)$ is a function from $(0,1)$ to $[0,\alpha]$.
    Let us verify the inequalities.
    First, \eqref{i:superadd} is a consequence of superadditivity of $f$:
    \begin{equation*}
        \frac{f(u, \lambda\theta u)}{u} \geq \frac{f(u, \theta u)}{u} + \theta\frac{f(\theta u, \lambda\theta u)}{\theta u}
    \end{equation*}
    and taking a limit infimum in $u$ yields the desired inequality.

    Next, \eqref{i:conv} is a consequence of the horizontal Lipschitz property.
    Fix $\theta > 0$.
    Let $\varepsilon>0$ and, by definition of $\Lambda(f)$, get arbitrarily large $u>0$ so that
    \begin{equation*}
        f(\theta u, \lambda\theta u)\leq \theta u (\varphi(\lambda) + \varepsilon).
    \end{equation*}
    Then by the $\alpha$-Lipschitz property,
    \begin{equation*}
        f(u, \lambda\theta u) \leq \alpha(1-\theta)u + f(\theta u, \lambda\theta u) \leq \alpha(1-\theta)u + u\theta\varphi(\lambda) + u\theta\varepsilon.
    \end{equation*}
    Since this holds for arbitrarily large $u$, dividing by $u$ and taking a limit infimum,
    \begin{equation*}
        \varphi(\lambda\theta) \leq (1-\theta)\alpha + \theta \varphi(\lambda) + \theta\varepsilon.
    \end{equation*}
    Since $\varepsilon>0$ was arbitrary, \eqref{i:conv} follows.

    Conversely, let $\varphi\in\mathcal{H}(\alpha)$ and define $f\colon\Delta_1 \to [0,\infty)$ by the rule $f(u,v) = u\varphi(v/u)$.
    We first check that $f\in\mathcal{L}_1(\alpha)$.
    Let $1\leq v \leq w \leq u$ be arbitrary, let $\theta = w/u$, and let $\lambda = v/w$.
    Then using \eqref{i:superadd},
    \begin{equation*}
        f(u,v) = u\varphi(\lambda \theta) \geq u\varphi(\theta) + u\theta\varphi(\lambda) = f(u, w)  + f(w, v),
    \end{equation*}
    yielding superadditivity of $f$.
    To verify the horizontal Lipschitz property, we again compute by \eqref{i:conv} that
    \begin{equation*}
        f(u, v) = u\varphi(\lambda \theta) \leq u\alpha(1-\theta) + u\theta \varphi(\lambda) = \alpha(u-w) + f(w,v)
    \end{equation*}
    which, after rearranging, shows that $u\mapsto f(u,v)$ is $\alpha$-Lipschitz.
    By \cref{l:domain-ext}, $\lex f\in\mathcal{L}(\alpha)$, and $\Lambda(\lex f)=\Lambda(f)=\varphi$.
    Therefore, $\Lambda$ is surjective.
\end{proof}

\subsection{Continuity and Lipschitz properties}\label{ss:cont-lip}
For the Assouad spectrum, the corresponding space of functions is the set $\mathcal{G}(\alpha)$ of functions $\varphi\colon(0,1)\to [0,\alpha]$ which are $\alpha$-Lipschitz and decreasing, satisfy $\lim_{\theta\nearrow1}\varphi(\theta)=0$, and obey the following analogue of \eqref{i:superadd} for all $\lambda,\theta\in (0,1)$:
\begin{equation*}
    \varphi(\lambda\theta) \leq \varphi(\theta) + \theta\varphi(\lambda).
\end{equation*}
When $\alpha = d$ is an integer, this is the family $\{\theta\mapsto (1-\theta)h(\theta) : h\in \mathcal{A}_d\}$ where $\mathcal{A}_d$ is defined in \cite{zbl:1562.28062}; and it is the same as the family $\mathcal{G}(\alpha)$ from \cite[Definition~1.2]{arxiv:2510.07013}.

A key difference in the setting of lower spectra is that the $\alpha$-Lipschitz property is replaced by property \eqref{i:conv}.
We note that this property is strictly weaker; see \cref{ss:unif-attain} for more discussion.
For now, let us note that the best Lipschitz property which we can obtain is the following.
\begin{lemma}\label{l:cont}
    Let $\alpha\geq 0$ and $\varphi\in\mathcal{H}(\alpha)$.
    Then for all $0<\lambda \leq 1$, $\varphi$ is $\lambda^{-1}(\alpha - \varphi(\lambda))$-Lipschitz on the interval $[\lambda,1]$.
    In particular, $\varphi$ is continuous on $(0,1]$.
\end{lemma}
\begin{proof}
    Fix $\theta\in[\lambda,1]$.
    By \eqref{i:conv}, for all $0<\theta'\leq\theta$, $\varphi(\theta')$ is bounded above by the line segment connecting $(0, \alpha)$ and $(\theta, \varphi(\theta))$.
    The negative of the slope of this segment is
    \begin{equation*}
        \frac{\alpha - \varphi(\theta)}{\theta} \leq \frac{\alpha - \varphi(\lambda)}{\lambda}.
    \end{equation*}
    Since $\varphi$ is decreasing, the other side of the Lipschitz inequality is trivial, completing the proof.
\end{proof}
\begin{remark}\label{r:cont-ext}
    Combined with the endpoint limits noted following \cref{d:H}, \cref{l:cont} shows that every $\varphi\in\mathcal{H}(\alpha)$ extends continuously to $[0,1]$ with $\varphi(1)=0$.
    We will therefore write $\varphi(0)$ and $\varphi(1)$ for the above values when convenient.
\end{remark}

\subsection{Classification of lower spectra}\label{ss:equiv-l}
In this section, we complete the proof of \cref{it:attain}.

One direction of the proof requires showing that the lower spectrum of $X$ (after normalization by a factor $(1-\theta)$) is an element of $\mathcal{H}(\alpha)$.
In the introduction, we introduced the equivalence relation $\sim$ on functions from $\Delta$ to $[0,\infty)$.
Below, we introduce a somewhat more general equivalence relation since it more precisely captures the type of errors which we are able to ignore.
\begin{definition}\label{d:eq}
    Let $f,g\colon\Delta\to[0,\infty)$ be arbitrary functions.
    We say that $f \asymp g$ if there exists a function $\eta\colon[0,\infty)\to[0,\infty)$ with $\lim_{u\to\infty}u^{-1}\eta(u) = 0$ such that for all $0 \leq v \leq u$:
    \begin{itemize}[nl]
        \item $f(u, v+\eta(u)) -\eta(u) \leq g(u,v)$ whenever $v \leq u-\eta(u)$, and
        \item $g(u,v) \leq f(u, v-\eta(u)) + \eta(u)$ whenever $v \geq \eta(u)$.
    \end{itemize}
\end{definition}
\begin{example}
    If $f,g\in\mathcal{L}$ and $|f(u, k) - g(u, k)| \leq \eta(u)$ for all $0 \leq k \leq u$ with $k\in\Z$ and $\lim_{u\to\infty}u^{-1}\eta(u) = 0$, then $f\asymp g$ with function $1 + \eta(u)$.
\end{example}
We first observe that the normalized limit $\Lambda$ is unchanged by the above equivalence relation.
\begin{lemma}\label{l:L-eq}
    Suppose $f\in\mathcal{L}(\alpha)$ and $g\colon\Delta \to [0,\infty)$ are such that $f\asymp g$.
    Then $\Lambda(f) = \Lambda(g)$.
\end{lemma}
\begin{proof}
    Let $\eta\colon[0,\infty)\to[0,\infty)$ be the function implicit in the relation $f\asymp g$, and we recall that $\lim_{u\to\infty}u^{-1}\eta(u) = 0$.
    Also, let $\varphi = \Lambda(f) \in \mathcal{H}(\alpha)$.

    Fix $\theta\in(0,1)$ and write $\varepsilon(u) = u^{-1}\eta(u)$.
    Since $f\asymp g$, for all $u > 0$ sufficiently large
    \begin{equation*}
        f(u, (\theta + \varepsilon(u))u) - \eta(u) \leq g(u, \theta u) \leq f(u, (\theta - \varepsilon(u)) u) + \eta(u).
    \end{equation*}
    Since $f$ is decreasing in the second variable and $\varphi$ is continuous by \cref{l:cont}, dividing by $u$ and taking a limit infimum gives that $\Lambda(g)(\theta) = \varphi(\theta)$.
\end{proof}
Now, recall that the usual definition of the lower spectrum from the introduction uses covers instead of packings.
We also recall the equivalence of packings and covers from \cref{e:cov-eq}.
We obtain the following alternative formula for $\dimLs\theta F$.

In the below statement, even though $\dimLs\theta F$ is \emph{a priori} not defined for $\theta = 1$, since $\dimLs\theta F \leq d$ for $F\subset\R^d$, it makes sense to write $(1-\theta)\dimLs\theta F$ to mean $0$ when $\theta = 1$.
\begin{corollary}\label{c:dimL-lim}
    Suppose $F\subset\R^d$.
    Then for all $\theta\in (0,1]$,
    \begin{equation*}
        (1-\theta)\dimLs\theta F = \Lambda(\lb_F)(\theta).
    \end{equation*}
    In particular, $\theta\mapsto(1-\theta)\dimLs\theta F \in \mathcal{H}(d)$.
\end{corollary}
\begin{proof}
    Define the function
    \begin{equation*}
        \gamma(u,v) = \log \inf_{x\in F} N_{2^{-u}}(F \cap B(x, 2^{-v})).
    \end{equation*}
    By \cref{e:cov-eq} and \cref{l:approx-br}, $\gamma\asymp \lb_F$, and by \cref{it:two-scale}~\cref{i:fL}, there is a function $f\in\mathcal{L}(d)$ such that $\lb_F \asymp f$.
    Therefore $\gamma\asymp f$ and thus
    \begin{equation*}
        \Lambda(\gamma) = \Lambda(f) = \Lambda(\lb_F)
    \end{equation*}
    by \cref{l:L-eq}.
    Moreover, $\Lambda(f)\in\mathcal{H}(d)$ by \cref{l:L-in-H}, and therefore $\Lambda(\lb_F)\in\mathcal{H}(d)$.

    It remains to show that $(1-\theta)\dimLs\theta F = \Lambda(\gamma)(\theta)$ for $\theta\in (0,1)$.
    Fix $\theta\in(0,1)$.
    Unpacking definitions, $\dimLs\theta F$ is equivalently the supremum over values $s\geq 0$ for which there exists a constant $m_s \geq 0$ such that for all $u\geq 0$,
    \begin{equation}\label{e:diml-equiv}
        \gamma(u, \theta u) \geq s(1-\theta) u - m_s.
    \end{equation}
    On one hand, if $s \geq 0$ and $m_s \geq 0$ satisfy \cref{e:diml-equiv}, dividing by $u$ and taking a limit infimum, it is immediate that $\Lambda(\gamma)(\theta) \geq s(1-\theta)$.
    On the other hand, suppose $t$ is such that for all $m \geq 0$ there is a $u_m\geq 0$ so that
    \begin{equation*}
        \gamma(u_m, \theta u_m) \leq t(1-\theta) u_m - m.
    \end{equation*}
    Since $\gamma$ is non-negative, we must have $\lim_{m\to\infty}u_m = \infty$ and therefore $\Lambda(\gamma)(\theta) \leq t(1-\theta)$.
    We conclude that $(1-\theta)\dimLs\theta F = \Lambda(\gamma)(\theta)$.
\end{proof}
We can now complete the proof of our main result.
\begin{restatement}{it:attain}
    Let $d\in\N$ and $\varphi\colon(0,1)\to[0,d]$.
    Then there exists a set $F\subset\R^d$ such that $(1-\theta)\dimLs\theta F = \varphi(\theta)$ for all $\theta\in(0,1)$ if and only if $\varphi\in\mathcal{H}(d)$.
\end{restatement}
\begin{proof}
    Recall that we saw in \cref{c:dimL-lim} that $\theta\mapsto (1-\theta)\dimLs\theta F \in\mathcal{H}(d)$, so we proceed with the converse.
    Fix $\varphi\in\mathcal{H}(d)$, and using surjectivity of $\Lambda$ from \cref{l:L-in-H} followed by \cref{it:two-scale}~\cref{i:ext}, we get a set $F\subset\R^d$ such that $\Lambda(f) = \varphi$ and $f\asymp \lb_F$.
    By construction, $\Lambda(f) = \varphi$, and by \cref{l:L-eq}, $\Lambda(f) = \Lambda(\lb_F)$.
    Therefore $(1-\theta)\dimLs\theta F = \varphi(\theta)$ for all $\theta\in(0,1)$, as required.
\end{proof}
\begin{remark}
    The set $F$ constructed in \cref{it:attain} can in fact be taken to be a countable union of non-trivial closed dyadic cubes.
    To prove this, one can modify the function $f$ as follows.
    Let $(n_k)_{k=0}^\infty$ be any sequence with $n_k \geq k+1$ such that $\lim_{k\to\infty}k/n_k = 0$.
    For each $d$-Lipschitz horizontal strip $f_k$, define a new $d$-Lipschitz function $g_k\in\mathcal{Y}(d)$ by the rule
    \begin{equation*}
        g_k(u) = \begin{cases}
            f_k(u) &: 0 \leq u \leq n_k\\
            f_k(n_k) + d(u - n_k) &: u \geq n_k
        \end{cases}
    \end{equation*}
    Since $g_k$ has constant slope $d$ on $[n_k,\infty)$, the corresponding set provided by \cref{ss:unif-ct} will be a finite union of dyadic cubes.
    Now, set
    \begin{equation*}
        \phi_k(u,v) =
        \begin{cases}
            g_k(u) - g_k(v) &: k \leq v,\\*
            d(u-v) &: k > v.
        \end{cases}
    \end{equation*}
    These are the maximal functions in $\mathcal{L}(d)$ which coincide with $g_k$ on the horizontal strip $\{(u,k): u \geq k\}\subset \Delta$.
    By definition, $\phi_k(u, k) = f(u, k)$ for $u \leq n_k$.
    We now show that $\phi_k(u, v)\geq f(u,v)$.
    Certainly, $\phi_k(u, v) \geq f(u, v)$ for $v \leq k$.
    Otherwise, consider $k \leq v \leq u$.
    \begin{itemize}[nl]
        \item If $n_k \leq v\leq u$, then by the horizontal $d$-Lipschitz property, $g_k(u) - g_k(v) = d(u-v) \geq f(u,v)$.
        \item If $v \leq n_k \leq u$, by superadditivity and the horizontal $d$-Lipschitz property,
            \begin{equation*}
                g_k(u) - g_k(v) = d(u - n_k) + f_k(n_k) - f_k(v) \geq d(u - n_k) + f(n_k, v) \geq f(u, v).
            \end{equation*}
        \item If $v \leq u \leq n_k$, by superadditivity
            \begin{equation*}
                g_k(u) - g_k(v) = f_k(u) - f_k(v) \geq f(u, v).
            \end{equation*}
    \end{itemize}
    Therefore, $g = \min_{k\geq0}\phi_k\in\mathcal{L}(d)$ and $g\geq f$.
    For each fixed $0<\theta<1$, for all sufficiently large $u$ and setting $k=\lfloor\theta u\rfloor$,
    \begin{equation*}
        f(u,\theta u)\leq g(u,\theta u)\leq f(u,\lfloor\theta u\rfloor).
    \end{equation*}
    By monotonicity and continuity of $\Lambda(f)$, it follows that $\Lambda(g)=\Lambda(f)$.
    Moreover, for each integer $k\geq0$, the function $u\mapsto g(u,k)$ eventually has slope $d$.
    The construction in the proof of \cref{it:two-scale}~\cref{i:ext} therefore gives a countable union of non-trivial closed dyadic cubes with the desired lower spectrum.
\end{remark}

\subsection{Monotone lower spectra}\label{ss:mono}
Recall that the monotone lower spectrum was defined in \cref{ss:add}, and that it is related to the usual spectrum by
\begin{equation}\label{e:vari2}
    \dimuLs\theta X = \inf_{0<\lambda\leq \theta}\dimLs\lambda X
\end{equation}
for $0<\theta < 1$.
This relationship allows us to provide a classification of the functions $\theta\mapsto (1-\theta)\dimuLs\theta X$, where $X\subset\R^d$.
The appropriate space of functions is the following.
\begin{definition}\label{d:M}
    Let $\alpha \geq 0$.
    We let $\mathcal{M}(\alpha)$ denote the functions $\varphi\colon(0,1)\to[0,\alpha]$ such that \eqref{i:conv} holds and the following condition holds:
    \begin{enumerate}[nl,a]
        \labelitem{i:mono}{M} \emph{Monotone:} The function $\theta\mapsto \varphi(\theta)/(1-\theta)$ is decreasing.
    \end{enumerate}
\end{definition}
The main difference from the space $\mathcal{H}(\alpha)$ is that the superadditivity assumption \eqref{i:superadd} is replaced by the monotonicity assumption \eqref{i:mono}.
We observe that the monotonicity assumption is stronger.
\begin{lemma}
    Let $\alpha \geq 0$.
    Then $\mathcal{M}(\alpha)\subset\mathcal{H}(\alpha)$ and the inclusion is strict for $\alpha > 0$.
\end{lemma}
\begin{proof}
    Let $\lambda,\theta\in(0,1)$.
    By \eqref{i:mono},
    \begin{equation*}
        \varphi(\theta) \leq \frac{\varphi(\lambda\theta)}{1-\lambda\theta}(1-\theta)
    \end{equation*}
    which, after rearranging and again applying \eqref{i:mono}, yields
    \begin{equation*}
        \frac{\varphi(\lambda\theta) - \varphi(\theta)}{\theta(1 - \lambda)} \geq \frac{\varphi(\lambda\theta)}{1-\lambda\theta} \geq \frac{\varphi(\lambda)}{1-\lambda}.
    \end{equation*}
    Rearranging, we obtain the required inequality \eqref{i:superadd}.

    To see that the inclusion is strict for $\alpha > 0$, see the example constructed in \cref{ss:non-mono} below.
\end{proof}
\begin{remark}\label{r:geom}
    By the preceding lemma and recalling \cref{l:cont}, each $\varphi\in\mathcal{M}(\alpha)$ extends continuously to $[0,1]$ with $\varphi(1)=0$.
    Property \eqref{i:mono} has a useful geometric interpretation: it says that for any $0<\theta \leq 1$, $\varphi$ is bounded above on $[\theta, 1]$ by the line segment joining $(\theta, \varphi(\theta))$ and $(1,0)$.

    We also recall the geometric interpretation of \eqref{i:conv}: for any $0<\theta \leq 1$, $\varphi$ is bounded above on $[0,\theta]$ by the line segment joining $(0,\alpha)$ and $(\theta, \varphi(\theta))$.
\end{remark}
\begin{remark}
    There is a natural geometric interpretation of the above computation.
    Rearranging \eqref{i:superadd}, for $0<\lambda<1$ and $0<\theta\leq 1$,
    \begin{equation*}
        \frac{\varphi(\lambda\theta) - \varphi(\theta)}{\theta - \lambda\theta} \geq \frac{\varphi(\lambda)}{1-\lambda}.
    \end{equation*}
    The left hand side is the negative of the slope of the line segment from $(\lambda\theta, \varphi(\lambda\theta))$ to $(\theta, \varphi(\theta))$, and the right hand side is the negative of the slope of the line segment from $(\lambda, \varphi(\lambda))$ to $(1, 0)$.
    Recalling \cref{r:geom}, we see that the monotonicity assumption \eqref{i:mono} is stronger.
\end{remark}
As the final component of our classification, we introduce a special class of functions which we can use to give an alternative definition of $\mathcal{M}(\alpha)$.
Given $\alpha \geq 0$, $\lambda\in(0,1)$, and $0 \leq t \leq \alpha(1-\lambda)$, we define the function
\begin{equation*}
    \phi_{\alpha, \lambda, t}(\theta) = \begin{cases}
        \alpha + \frac{\theta}{\lambda}(t-\alpha) &: 0 < \theta \leq \lambda,\\*
        t\frac{1-\theta}{1-\lambda} &: \lambda < \theta < 1.
    \end{cases}
\end{equation*}
This is the function with two affine parts connecting the points $(0, \alpha)$, $(\lambda, t)$, and $(1, 0)$.
The constraint on $t$ ensures that the negative of the slope of $\phi_{\alpha, \lambda, t}$ on $[0,\lambda]$ is greater than or equal to the negative of the slope on $[\lambda, 1]$.
In particular, with \cref{r:geom} in mind, the following lemma is immediate.
\begin{lemma}\label{l:hM}
    Let $\alpha \geq 0$.
    Then $\phi_{\alpha, \lambda, t}\in\mathcal{M}(\alpha)$ for all $\lambda\in(0,1)$ and $0 \leq t \leq \alpha(1-\lambda)$.
    Moreover, a function $\varphi\colon(0,1)\to[0,\alpha]$ is an element of $\mathcal{M}(\alpha)$ if and only if, for every $\lambda\in(0,1)$, $\varphi(\lambda)\leq\alpha(1-\lambda)$ and $\varphi\leq\phi_{\alpha,\lambda,\varphi(\lambda)}$.
\end{lemma}
We can now state and prove our main result for the monotone lower spectrum.
In particular, this implies \cref{ic:mono-class}.
\begin{theorem}\label{t:mono-class}
    Let $d\in\N$ and $\varphi\colon(0,1)\to[0,d]$.
    Then the following are equivalent.
    \begin{enumerate}[nl,r]
        \item\label{i:in-M} $\varphi\in\mathcal{M}(d)$.
        \item\label{i:dim} There exists $F\subset\R^d$ such that $(1-\theta)\dimuLs\theta F = \varphi(\theta)$ for all $\theta\in(0,1)$.
        \item\label{i:inf} $\varphi$ is an infimum of a family of functions of the form $\phi_{d, \lambda, t}$ for $\lambda\in(0,1)$ and $0\leq t \leq d(1-\lambda)$.
    \end{enumerate}
\end{theorem}
\begin{proof}
    To see that \cref{i:in-M} implies \cref{i:dim}, since $\mathcal{M}(d)\subset\mathcal{H}(d)$, we may apply \cref{it:attain} to get a set $F\subset\R^d$ such that $(1-\theta)\dimLs\theta F = \varphi(\theta)$.
    Then by \cref{e:vari2}, since $\theta\mapsto\varphi(\theta)/(1-\theta)$ is decreasing, it follows that $\dimuLs\theta F = \dimLs\theta F$, as required.

    To see that \cref{i:dim} implies \cref{i:inf}, write $\beta(\theta) = (1-\theta)\dimLs\theta F$ so $\beta\in\mathcal{H}(d)$.
    Then, observe for all $0<\theta < 1$ that
    \begin{equation*}
        (1-\theta)\inf_{0<\lambda\leq \theta}\frac{\beta(\lambda)}{1-\lambda} = \inf_{0<\lambda \leq \theta}\phi_{d, \lambda, \beta(\lambda)}(\theta) = \inf_{0<\lambda < 1}\phi_{d, \lambda, \beta(\lambda)}(\theta).
    \end{equation*}
    But by \cref{e:vari2}, the function on the left is precisely $(1-\theta)\dimuLs\theta F = \varphi(\theta)$.

    Finally, \cref{i:inf} implies \cref{i:in-M} since $\phi_{d, \lambda, t}\in\mathcal{M}(d)$ and it is easy to check that $\mathcal{M}(d)$ is closed under pointwise infima.
\end{proof}

\subsection{Two examples of lower spectra}
We consider two explicit families of examples of lower spectra to highlight some of the possible behaviour.

\subsubsection{Convex functions}
Suppose $f\colon [0,1]\to [0,\alpha]$ is decreasing, continuous, and convex with $f(0) \leq \alpha$ and $f(1) = 0$.
For $0<\theta \leq 1$, applying convexity on the interval $[0, \theta]$, we see that $f$ satisfies \eqref{i:conv}, and applying convexity on the interval $[\theta, 1]$, we see that $f$ satisfies \eqref{i:mono}.
Therefore $f \in \mathcal{M}(\alpha)$.
This shows, for example, that the Lipschitz property proven in \cref{l:cont} is essentially optimal, and demonstrates that there are many sets $F \subset \R^d$ for which $\theta \mapsto (1-\theta)\dimLs\theta F$ and $\theta \mapsto \dimLs\theta F$ are not Lipschitz on $(0,1)$.

However, not all functions in $\mathcal{M}(\alpha)$ are of this form.
One can construct non-convex examples by taking an infimum of functions of the form $\phi_{\alpha, \lambda, t}$ from \cref{ss:mono}.
See \cref{f:mini} for a depiction.
\begin{figure}[t]
    \centering
    \begin{subcaptionblock}{.47\textwidth}
        \centering
        \begin{tikzpicture}[x=6cm, y=4cm]
  \pgfmathdeclarefunction{gfun}{1}{\pgfmathparse{(1-#1)^4}}

  \newcommand{\drawh}[3][]{%
    \pgfmathsetmacro{\gy}{gfun(#3)}%
    \draw[#1] (0,#2) -- (#3,\gy) -- (1,0);%
  }

  \pgfmathsetmacro{\aI}{1.0}    \pgfmathsetmacro{\yI}{0.15}
  \pgfmathsetmacro{\aII}{1.0}   \pgfmathsetmacro{\yII}{0.4}
  \pgfmathsetmacro{\aIII}{1.0}  \pgfmathsetmacro{\yIII}{0.65}

  \pgfmathsetmacro{\gI}{gfun(\yI)}
  \pgfmathsetmacro{\gII}{gfun(\yII)}
  \pgfmathsetmacro{\gIII}{gfun(\yIII)}

  \draw[-stealth] (-0.05,0) -- (1.05,0);
  \draw[-stealth] (0,-0.03) node[below, font=\small]{$0$}-- (0,{\aIII+0.04});

  \drawh[gray, thin]{\aI}{\yI}
  \drawh[gray, thin]{\aII}{\yII}
  \drawh[gray, thin]{\aIII}{\yIII}

  \pgfmathsetmacro{\sI}{\gI/(1-\yI)}
  \pgfmathsetmacro{\mIIa}{(\gII-\aII)/\yII}
  \pgfmathsetmacro{\tAB}{(\sI-\aII)/(\mIIa+\sI)}
  \pgfmathsetmacro{\vAB}{\aII+\mIIa*\tAB}

  \pgfmathsetmacro{\sII}{\gII/(1-\yII)}
  \pgfmathsetmacro{\mIIIa}{(\gIII-\aIII)/\yIII}
  \pgfmathsetmacro{\tBC}{(\sII-\aIII)/(\mIIIa+\sII)}
  \pgfmathsetmacro{\vBC}{\aIII+\mIIIa*\tBC}

  \draw[black, very thick]
    (0,\aI) -- (\yI,\gI)
    -- (\tAB,\vAB)
    -- (\yII,\gII)
    -- (\tBC,\vBC)
    -- (\yIII,\gIII)
    -- (1,0);

  \draw[densely dotted] (\yI,\gI) -- (0,{\gI + \yI*(\vAB-\gI)/(\yI-\tAB)}) node[left,font=\small] {$\kappa_1$};
  \draw[densely dotted] (\yII,\gII) -- (0,{\gII + \yII*(\vBC-\gII)/(\yII-\tBC)}) node[left,font=\small] {$\kappa_2$};
  \draw[densely dotted] (\yIII,\gIII) -- (0,{\gIII + \yIII*(\gIII)/(1-\yIII)}) node[left,font=\small] {$\kappa_3$};

  \draw[densely dotted] (\yI,\gI) -- (\yI,-0.03)node[below,font=\small] {$\lambda_1$};
  \draw[densely dotted] (\yII,\gII) -- (\yII,-0.03)node[below,font=\small] {$\lambda_2$};
  \draw[densely dotted] (\yIII,\gIII) -- (\yIII,-0.03)node[below,font=\small] {$\lambda_3$};
  \draw (1,0.02) -- (1,-0.03)node[below,font=\small] {$1$};
  \node[left] at (0, 1) {$\alpha$};
\end{tikzpicture}
        \caption{The functions $\phi_{\alpha,\lambda_i, t_i}$}
    \end{subcaptionblock}
    \begin{subcaptionblock}{.47\textwidth}
        \centering
        \begin{tikzpicture}[x=6cm, y=4cm]
  \pgfmathdeclarefunction{gfun}{1}{\pgfmathparse{(1-#1)^4}}

  \newcommand{\drawh}[3][]{%
    \pgfmathsetmacro{\gy}{gfun(#3)}%
    \pgfmathsetmacro{\mh}{(\gy - #2) / #3}%
    \pgfmathsetmacro{\ch}{\gy / (1 - #3)}%
    \draw[#1]
      plot[domain=0:#3,samples=60,smooth] (\x, {(\mh*\x + #2)/(1-\x)})
      -- plot[domain=#3:1,samples=2] (\x, {\ch});%
  }

  \pgfmathsetmacro{\aI}{1.0}    \pgfmathsetmacro{\yI}{0.15}
  \pgfmathsetmacro{\aII}{1.0}   \pgfmathsetmacro{\yII}{0.4}
  \pgfmathsetmacro{\aIII}{1.0}  \pgfmathsetmacro{\yIII}{0.65}

  \pgfmathsetmacro{\gI}{gfun(\yI)}
  \pgfmathsetmacro{\gII}{gfun(\yII)}
  \pgfmathsetmacro{\gIII}{gfun(\yIII)}

  \draw[-stealth] (-0.05,0) -- (1.05,0);
  \draw[-stealth] (0,-0.03) node[below, font=\small]{$0$}-- (0,{\aIII+0.04});

  \drawh[gray, thin]{\aI}{\yI}
  \drawh[gray, thin]{\aII}{\yII}
  \drawh[gray, thin]{\aIII}{\yIII}

  \pgfmathsetmacro{\mIa}{(\gI-\aI)/\yI}
  \pgfmathsetmacro{\mIIa}{(\gII-\aII)/\yII}
  \pgfmathsetmacro{\mIIIa}{(\gIII-\aIII)/\yIII}

  \pgfmathsetmacro{\sI}{\gI/(1-\yI)}
  \pgfmathsetmacro{\sII}{\gII/(1-\yII)}
  \pgfmathsetmacro{\sIII}{\gIII/(1-\yIII)}

  \pgfmathsetmacro{\tAB}{(\sI-\aII)/(\mIIa+\sI)}

  \pgfmathsetmacro{\tBC}{(\sII-\aIII)/(\mIIIa+\sII)}

  \pgfmathsetmacro{\eps}{0.0014}
  \pgfmathsetmacro{\epsa}{0.00095}
  \draw[black, very thick]
    plot[domain=0:\yI,samples=60,smooth] (\x, {(\mIa*\x + \aI)/(1-\x)})
    -- ({\tAB+\eps},\sI)
    plot[domain={\tAB-\epsa}:\yII,samples=60,smooth] (\x, {(\mIIa*\x + \aII)/(1-\x)})
    -- ({\tBC+\eps},\sII)
    plot[domain={\tBC-\epsa}:\yIII,samples=60,smooth] (\x, {(\mIIIa*\x + \aIII)/(1-\x)})
    -- (1,\sIII);

  \draw[densely dotted] (\yI,\sI) -- (0, \sI) node[left,font=\small]{$\kappa_1$};
  \draw[densely dotted] (\yII,\sII) -- (0, \sII) node[left,font=\small]{$\kappa_2$};
  \draw[densely dotted] (\yIII,\sIII) -- (0, \sIII) node[left,font=\small]{$\kappa_3$};

  \draw[densely dotted] (\yI,\sI) -- (\yI,-0.03) node[below,font=\small]{$\lambda_1$};
  \draw[densely dotted] (\yII,\sII) -- (\yII,-0.03) node[below,font=\small]{$\lambda_2$};
  \draw[densely dotted] (\yIII,\sIII) -- (\yIII,-0.03) node[below,font=\small]{$\lambda_3$};
  \draw (1,0.02) -- (1,-0.03)node[below,font=\small] {$1$};

  \node[left] at (0, 1) {$\alpha$};
\end{tikzpicture}
        \caption{The functions $\theta\mapsto \phi_{\alpha,\lambda_i, t_i}(\theta)/(1-\theta)$}
    \end{subcaptionblock}
    \caption{A function obtained as the minimum of three functions $\phi_{\alpha, \lambda_i, t_i}$, defined in \cref{ss:mono}, and the corresponding lower spectrum.
        Here, $\kappa_i = t_i/(1-\lambda_i)$.
    }
    \label{f:mini}
\end{figure}

\subsubsection{Non-monotonic lower spectrum}\label{ss:non-mono}
In this section, we give a simple example of a set with non-monotonic lower spectrum, different to the example from \cite[§8]{zbl:1390.28019}.
The construction is similar in spirit to the family of examples in \cite[§3.3.2]{zbl:1562.28062}.

The example will be an element of $\mathcal{H}(\alpha)$ for $\alpha\geq 0$, and it is parametrized by numbers $0<a_1<a_2\leq 1$ and a slope $0 \leq \kappa \leq \alpha$.
Write $\bm{a} = (a_1, a_2,\kappa)$.
The function $q_{\bm{a}}$ is defined so that it is continuous, has slope $-\kappa$ on the intervals $[a_1a_2, a_1]$ and $[a_2, 1]$, and satisfies inequality \eqref{i:conv} with equality on the intervals $[0,a_1a_2]$ and $[a_1,a_2]$.
The precise definition is a bit tedious to state.
Set
\begin{align*}
    \alpha_1 &= \frac{\alpha - \kappa(a_1 - a_1a_2 + 1-a_2) - \alpha_2(a_2-a_1)}{a_1a_2}\\
    \alpha_2 &= \frac{\alpha - \kappa(1-a_2)}{a_2}
\end{align*}
and then set
\begin{equation*}
    q_{\bm{a}}(\theta) =
    \begin{cases}
        \alpha_1(a_1a_2 - \theta) + \kappa(a_1 - a_1a_2 + 1-a_2) + \alpha_2(a_2 - a_1) &: 0 \leq \theta \leq a_1a_2\\
        \kappa(a_1 - \theta + 1-a_2) + \alpha_2(a_2 - a_1) &: a_1a_2 \leq \theta \leq a_1\\
        \kappa(1-a_2) + \alpha_2(a_2 - \theta) &: a_1 \leq \theta \leq a_2\\
        \kappa(1-\theta) &: a_2 \leq \theta \leq 1
    \end{cases}
\end{equation*}
See \cref{f:non-mono} for a depiction of the function $q_{\bm{a}}$.
Note that $q_{\bm{a}}$ has slope $-\alpha_1$ on $[0, a_1a_2]$ and slope $-\alpha_2$ on $[a_1, a_2]$.
In general, $\kappa\leq \alpha_2\leq \alpha_1$, and all of the inequalities are strict if $\kappa < \alpha$ and $a_2 < 1$.
Also, it is easy to see that $q_{\bm{a}}$ satisfies \eqref{i:conv}.
\begin{figure}[t]
    \begin{subcaptionblock}{.47\textwidth}
        \centering
        \begin{tikzpicture}[x=6cm, y=4cm]
  \pgfmathsetmacro{\aI}{1/2}       %
  \pgfmathsetmacro{\aII}{2/3}      %
  \pgfmathsetmacro{\kap}{1/4}      %
  \pgfmathsetmacro{\alp}{1}        %

  \pgfmathsetmacro{\aIaII}{\aI*\aII}                    %
  \pgfmathsetmacro{\faII}{\kap*(1-\aII)}                %
  \pgfmathsetmacro{\faI}{\alp+\aI*(\faII-\alp)/\aII}   %
  \pgfmathsetmacro{\faIaII}{\faI+\kap*(\aI-\aIaII)}    %

  \draw[-stealth] (-0.05,0) -- (1.08,0);
  \draw[-stealth] (0,-0.03) node[below,font=\small] {$0$} -- (0,{\alp+0.12});

  \draw[densely dashed, gray] (0,\alp) -- (\aII,\faII);

  \draw[thick] (0,\alp) -- (\aIaII,\faIaII)
                         -- (\aI,\faI)
                         -- (\aII,\faII)
                         -- (1,0);

  \draw[densely dotted] (\aIaII,\faIaII) -- (\aIaII,-0.03) node[below,font=\small] {$a_1a_2$};
  \draw[densely dotted] (\aI,\faI) -- (\aI,-0.03) node[below,font=\small] {$a_1$};
  \draw[densely dotted] (\aII,\faII) -- (\aII,-0.03) node[below,font=\small]{$a_2$};
  \draw (1,0.02) -- (1,-0.03) node[below,font=\small]{$1$};

  \node[left,font=\small] at (0,\alp) {$\alpha$};
\end{tikzpicture}
        \caption{Plot of $q_{\bm{a}}$.}
    \end{subcaptionblock}
    \begin{subcaptionblock}{.47\textwidth}
        \centering
        \begin{tikzpicture}[x=6cm, y=4cm]
  \pgfmathsetmacro{\aI}{1/2}       %
  \pgfmathsetmacro{\aII}{2/3}      %
  \pgfmathsetmacro{\kap}{1/4}      %
  \pgfmathsetmacro{\alp}{1}        %

  \pgfmathsetmacro{\aIaII}{\aI*\aII}                    %
  \pgfmathsetmacro{\faII}{\kap*(1-\aII)}                %
  \pgfmathsetmacro{\faI}{\alp+\aI*(\faII-\alp)/\aII}   %
  \pgfmathsetmacro{\faIaII}{\faI+\kap*(\aI-\aIaII)}    %

  \pgfmathsetmacro{\gaIaII}{\faIaII/(1-\aIaII)}
  \pgfmathsetmacro{\gaI}{\faI/(1-\aI)}

  \pgfmathsetmacro{\mI}{(\faIaII-\alp)/\aIaII}
  \pgfmathsetmacro{\nI}{\alp}
  \pgfmathsetmacro{\mII}{-\kap}
  \pgfmathsetmacro{\nII}{\faIaII+\kap*\aIaII}
  \pgfmathsetmacro{\mIII}{(\faII-\faI)/(\aII-\aI)}
  \pgfmathsetmacro{\nIII}{\faI-\mIII*\aI}
  \pgfmathsetmacro{\mRef}{(\faII-\alp)/\aII}

  \draw[-stealth] (-0.05,0) -- (1.08,0);
  \draw[-stealth] (0,-0.03) node[below,font=\small]{$0$} -- (0,{\alp+0.12});

  \draw[densely dashed, gray]
    plot[domain=0:\aII,samples=60,smooth] (\x, {(\mRef*\x+\alp)/(1-\x)});

  \draw[thick]
    plot[domain=0:\aIaII,samples=60,smooth] (\x, {(\mI*\x+\nI)/(1-\x)})
    -- plot[domain=\aIaII:\aI,samples=40,smooth] (\x, {(\mII*\x+\nII)/(1-\x)})
    -- plot[domain=\aI:\aII,samples=40,smooth] (\x, {(\mIII*\x+\nIII)/(1-\x)})
    -- plot[domain=\aII:1,samples=2] (\x, {\kap});

  \draw[densely dotted] (\aIaII,\gaIaII) -- (\aIaII,-0.03)node[below,font=\small] {$a_1a_2$};
  \draw[densely dotted] (\aI,\gaI) -- (\aI,-0.03)node[below,font=\small] {$a_1$};
  \draw[densely dotted] (\aII,\kap) -- (\aII,-0.03)node[below,font=\small] {$a_2$};
  \draw (1, 0.02) -- (1,-0.03)node[below,font=\small] {$1$};

  \draw[densely dotted] (0,\kap) -- (\aII,\kap);

  \node[left,font=\small] at (0,\alp) {$\alpha$};
  \node[left,font=\small] at (0,\kap) {$\kappa$};
\end{tikzpicture}
        \caption{Plot of $\theta\mapsto q_{\bm{a}}(\theta)/(1-\theta)$.}
    \end{subcaptionblock}
    \caption{A plot of the function $q_{\bm{a}}$ where $\alpha = 1$ and $\bm{a} = (1/2, 2/3, 1/4)$, as well as the corresponding lower spectrum.}
    \label{f:non-mono}
\end{figure}

In order to prove \eqref{i:superadd}, fix $\theta\in(0,1]$ and consider the function
\begin{equation*}
    \phi_{\bm{a}, \theta}(\lambda) = \frac{q_{\bm{a}}(\theta \lambda) - q_{\bm{a}}(\theta)}{\theta}.
\end{equation*}
We must show that $\phi_{\bm{a}, \theta} \geq q_{\bm{a}}$ for all $\theta\in(0,1]$.

To prove this, observe that $\phi_{\bm{a},\theta}$ acts as follows.
For $0<\theta \leq 1$, let
\begin{align*}
    I_1(\theta) &= [0, a_1a_2/\theta]\cap(0,1],\\
    I_2(\theta) &= [a_1a_2/\theta, a_1/\theta]\cap(0,1],\\
    I_3(\theta) &= [a_1/\theta, a_2/\theta]\cap(0,1],\\
    I_4(\theta) &= [a_2/\theta, 1]\cap(0,1].
\end{align*}
Of course, some of the intervals $I_j(\theta)$ will be empty for small $\theta$.
For an interval $I$, write $|I|$ to denote its length.
Then $\phi_{\bm{a},\theta}$ is the continuous function with slopes $-\alpha_1$ on $I_1(\theta)$, $-\kappa$ on $I_2(\theta)$, $-\alpha_2$ on $I_3(\theta)$, and $-\kappa$ on $I_4(\theta)$ such that $\phi_{\bm{a},\theta}(1) = 0$.

Recall that $\kappa\leq\alpha_2 \leq \alpha_1$.
We consider four cases depending on the value of $\theta\in(0,1]$:
\begin{enumerate}
    \item Suppose $a_2 < \theta$.
        Then $|I_j(\theta)| \geq |I_j(1)|$ for $j \leq 3$ and $|I_4(\theta)| \leq |I_4(1)|$ so that $q_{\bm{a}} \leq \phi_{\bm{a},\theta}$.
    \item Suppose $a_1 < \theta \leq a_2$.
        Then $|I_4(\theta)| = 0$, $|I_2(\theta)| \leq |I_4(1)|$, and $a_1 \leq a_1a_2/\theta < a_2$.
        Therefore, $q_{\bm{a}}(a_1a_2/\theta)\leq \phi_{\bm{a},\theta}(a_1a_2/\theta)$ and it follows that $q_{\bm{a}} \leq \phi_{\bm{a},\theta}$.
    \item Suppose $a_1a_2 \leq \theta \leq a_1$.
        Then $|I_3(\theta)| = |I_4(\theta)| = 0$ and $|I_2(\theta)| \leq |I_4(1)|$, so $q_{\bm{a}} \leq \phi_{\bm{a},\theta}$.
    \item Suppose $\theta < a_1a_2$.
        Then $\phi_{\bm{a},\theta}$ has constant slope $-\alpha_1$ and the result is trivial.
\end{enumerate}
Therefore $q_{\bm{a}}$ satisfies \eqref{i:superadd}, so that $q_{\bm{a}}\in\mathcal{H}(\alpha)$.

\section{Uniform sets and lower spectra}\label{s:unif}
We conclude this paper with our results concerning uniform sets.
This includes the construction of uniform sets used in the proof of \cref{it:two-scale}, as well as the proof of \cref{it:attain-u}.

\subsection{Uniform sets}
First, we recall the definition of the (upper) two-scale branching function from \cite{arxiv:2510.07013}.
For a metric space $X$ we let $\ub = \ub_X\colon \Delta \to [0, \infty]$ denote the function
\begin{equation*}
    \ub_X(u,v) = \log \sup_{x\in X}N_{2^{-u}}(B(x, 2^{-v})).
\end{equation*}
This is the upper variant of the branching function $\lb_X$.
The following properties of $\ub_X$ are easy to check:
\begin{enumerate}[nl,r]
    \item $\ub_X(u,u) = 0$ for all $u\in\R$.
    \item $\ub_X(u,v)$ is increasing in $u$ and decreasing in $v$.
    \item For all $v\leq w \leq u$,
        \begin{equation*}
            \ub_X(u,v) \leq \ub_X(u, w) + \ub_X(w, v).
        \end{equation*}
\end{enumerate}
For a general metric space $X$ the inequality $\lb_X \leq \ub_X$ always holds, and clearly these functions can be very different.
For much more detail concerning the function $\ub_X$, we refer the reader to \cite{arxiv:2510.07013}.
\begin{definition}\label{d:unif}
    Let $\eta\colon\R\to\R$.
    We say that a metric space $X$ is \emph{$\eta$-uniform} if $\lim_{u\to\infty}u^{-1}\eta(u) = 0$ and
    \begin{equation*}
        \ub_X(u,v) \leq \lb_X(u,v) + \eta(u)\quad\text{for all}\quad(u,v)\in\Delta.
    \end{equation*}
    We say that $X$ is \emph{uniform} if it is $\eta$-uniform for some $\eta$.
\end{definition}
As long as the space $X$ is somewhat nice (so that covers and packings are comparable), the two-scale branching functions of uniform sets are determined entirely by covering numbers.
\begin{definition}
    For $\alpha \geq 0$, we let $\mathcal{Y}(\alpha)$ denote the set of increasing $\alpha$-Lipschitz functions $f\colon[0,\infty)\to[0,\infty)$ with $f(0) = 0$.
\end{definition}
We now obtain our description of uniform sets.
\begin{corollary}\label{c:u-lip}
    Let $d\in\N$.
    Then there is a constant $C_d > 0$ so that for any $\eta$-uniform subset $X\subset\R^d$, there is a function $f\in\mathcal{Y}(d)$ such that for all $0 \leq v \leq u$,
    \begin{align*}
        |\ub_X(u,v) - (f(u) - f(v))| \leq \eta(u) + \eta(v) + C_d,\\*
        |\lb_X(u,v) - (f(u) - f(v))| \leq \eta(u) + \eta(v) + C_d.
    \end{align*}
\end{corollary}
\begin{proof}
    Let $X \subset \R^d$ be $\eta$-uniform and let $f_X\in\mathcal{L}(d)$ satisfy \cref{e:repl-bounds}.
    Set
    \begin{equation*}
        f(u) = f_X(u, 0)\quad\text{for all}\quad u\geq 0.
    \end{equation*}
    Since $f_X\in\mathcal{L}(d)$, $f$ is increasing and $d$-Lipschitz with $f(0) = 0$, so $f\in\mathcal{Y}(d)$.
    In $\R^d$, $\ub_X(v,v-1) = O_d(1)$, so by subadditivity and \cref{e:repl-bounds},
    \begin{equation*}
        f_X(u,v) \leq \lb_X(u,v-1) \leq \ub_X(u,v-1) \leq \ub_X(u,v) + O_d(1).
    \end{equation*}
    Now, by superadditivity of $f_X$, $\eta$-uniformity at scale $v$, and subadditivity of $\ub_X$,
    \begin{align*}
        \lb_X(u,v) - C_d &\leq f_X(u,v) \leq f(u) - f(v)\\
                         &\leq \ub_X(u,0) - \ub_X(v,0) + \eta(v) + O_d(1)\\
                         &\leq \ub_X(u,v) + \eta(v) + O_d(1).
    \end{align*}
    Using $\lb_X(u,v) \leq \ub_X(u,v) \leq \lb_X(u,v) + \eta(u)$ and increasing $C_d$ if necessary completes the proof.
\end{proof}
\begin{remark}
    Setting $v = 0$, we see that the function $f$ provided by \cref{c:u-lip} satisfies
    \begin{equation*}
        \ub_X(u, 0) = f(u) + o(u).
    \end{equation*}
    (This is also clear from the proof itself.)
    For example, if $\diam X \leq 1$, then $\ub_X(u,0) =\log N_{2^{-u}}(X)$ is just the logarithm of the covering number at scale $2^{-u}$.
    \cref{c:u-lip} says that these covering numbers approximately determine all of the two-scale covering numbers of a uniform set $X$.
\end{remark}

\subsection{Constructing uniform sets}\label{ss:unif-ct}
A standard way to construct a uniform set in $\R^d$ is as follows.
Suppose $f\in\mathcal{Y}(d)$ is arbitrary.
By taking a supremum over all integer-valued $d$-Lipschitz functions with $f(0) = 0$ bounded above by $f$, there exists a function $h\colon\Z\cap[0,\infty)\to\Z\cap[0,\infty)$ which is also an increasing and $d$-Lipschitz function with $h(0) = 0$, such that
\begin{equation*}
    f(k) - 1 < h(k) \leq f(k)\quad\text{for all}\quad k\geq 0.
\end{equation*}
Then, define a sequence $(a_k)_{k=1}^\infty \in \{0,\ldots,d\}^{\N}$ by the rule $a_k = h(k) - h(k-1)$.
We inductively construct nested families of dyadic cubes as follows.
Let $\mathcal{Q}_0 = \{[0,1]^d\}$.
Now, suppose we have constructed $\mathcal{Q}_k$ as a union of dyadic cubes at level $k$, for some $k \geq 0$.
Then, replace each cube $Q\in\mathcal{Q}_k$ with $2^{a_{k+1}}$ sub-cubes at level $k+1$ (the precise choice of sub-cube is irrelevant).
Finally, set
\begin{equation*}
    K_f = \bigcap_{k=0}^\infty\bigcup_{Q\in\mathcal{Q}_k}Q.
\end{equation*}
Since each level $k$ dyadic cube contains exactly $2^{h(n) - h(k)}$ dyadic cubes at level $n$ for $k \leq n$, it follows that $K_f$ is an $O_d(1)$-uniform set with branching function
\begin{equation*}
    \lb_{K_f}(u,v) = f(u) - f(v) + O_d(1).
\end{equation*}
Moreover, if $f(z) = 0$ for some $z \geq 0$, then $K_f$ is contained in a single dyadic cube of side-length $2^{-\lfloor z\rfloor}$.
Translating this cube to the origin and scaling by $2^{\lfloor z\rfloor-z}$ gives a set in $[0,2^{-z}]^d$ with the same branching function (up to an error term of order $O_d(1)$).

To summarize, we have proven the following.
\begin{proposition}\label{p:unif-ct}
    Let $d\in\N$ and $f\in\mathcal{Y}(d)$.
    Suppose $f(z) = 0$ for some $z\geq 0$.
    Then there exists a compact $O_d(1)$-uniform set $K_f\subset [0, 2^{-z}]^d$ such that
    \begin{equation*}
        \lb_{K_f}(u,v) = f(u) - f(v) + O_d(1)
    \end{equation*}
    for all $0\leq v\leq u$.
\end{proposition}

\subsection{Classification of lower spectra of uniform sets}\label{ss:unif-attain}
The goal of this section is to study the lower spectra of uniform sets and to prove \cref{it:attain-u}.
We begin with the space of functions.
\begin{definition}
    Let $\alpha \geq 0$.
    We let $\mathcal{U}(\alpha)$ denote the functions $\varphi\colon (0,1)\to [0,\alpha]$ which are $\alpha$-Lipschitz and satisfy \eqref{i:superadd}.
\end{definition}
Comparing the definition of $\mathcal{U}(\alpha)$ with $\mathcal{H}(\alpha)$, we see that we have replaced the convexity property \eqref{i:conv} with the $\alpha$-Lipschitz property.
The $\alpha$-Lipschitz property is strictly stronger.
\begin{lemma}\label{l:U-inc}
    For all $\alpha \geq 0$, we have $\mathcal{U}(\alpha) \subset \mathcal{H}(\alpha)$, and the inclusion is strict if $\alpha > 0$.
\end{lemma}
\begin{proof}
    Let $\varphi\in\mathcal{U}(\alpha)$.
    By the argument following \cref{d:H} (in \cref{i:iter}), \eqref{i:superadd} and boundedness give that $\varphi(\lambda)\leq\alpha(1-\lambda)$.
    Therefore, by the $\alpha$-Lipschitz property,
    \begin{equation*}
        \varphi(\lambda\theta) - \varphi(\lambda) \leq \alpha \lambda (1-\theta) \leq (1-\theta)(\alpha - \varphi(\lambda)).
    \end{equation*}
    This proves \eqref{i:conv}, and therefore $\varphi\in\mathcal{H}(\alpha)$.

    To see that the inclusion is strict when $\alpha > 0$, use for example the functions $\phi_{\alpha, \lambda, t}\in\mathcal{H}(\alpha)$ defined in \cref{ss:mono}.
\end{proof}
We now prove that $\mathcal{U}(d)$ is precisely the set of lower spectra of uniform subsets of $\R^d$.
\begin{restatement}{it:attain-u}
    Let $d\in\N$ and let $\varphi\colon(0,1)\to [0,d]$.
    Then there is a uniform set $F\subset\R^d$ with $(1-\theta)\dimLs\theta F = \varphi(\theta)$ for all $\theta\in(0,1)$ if and only if $\varphi\in\mathcal{U}(d)$.
\end{restatement}
\begin{proof}
    First suppose that $F\subset\R^d$ is $\eta$-uniform and write $\varphi(\theta) = (1-\theta)\dimLs\theta F$.
    By \cref{c:u-lip}, get $f\in\mathcal{Y}(d)$ and a constant $C_d\geq 0$ such that for all $0\leq v\leq u$,
    \begin{equation*}
        |\lb_F(u,v) - (f(u) - f(v))| \leq C_d + \eta(u) + \eta(v).
    \end{equation*}
    Thus by \cref{c:dimL-lim}, using the fact that $\theta > 0$,
    \begin{equation*}
        \varphi(\theta) = \liminf_{u\to\infty}\frac{f(u) - f(\theta u)}{u}.
    \end{equation*}
    Since $f$ is $d$-Lipschitz, $\theta\mapsto u^{-1}(f(u) - f(\theta u))$ is $d$-Lipschitz, so $\varphi$ is $d$-Lipschitz and therefore $\varphi\in\mathcal{U}(d)$.

    For the converse direction, for a given $f\in\mathcal{Y}(d)$, the corresponding uniform set $K_f\subset [0,1]^d$ provided by \cref{p:unif-ct} satisfies
    \begin{equation}\label{e:dimls-e}
        (1-\theta)\dimLs\theta K_f = \liminf_{u\to\infty}\frac{f(u) - f(\theta u)}{u}.
    \end{equation}
    Therefore it suffices to choose $f\in\mathcal{Y}(d)$ such that this limit infimum is exactly the function $\varphi\in\mathcal{U}(d)$.
    Let $0 = v_1 < u_1 < v_2 < u_2 <\cdots$ be a sequence such that
    \begin{equation*}
        \lim_{k\to\infty}\frac{v_k}{u_k} = \lim_{k\to\infty}\frac{u_k}{v_{k+1}} = 0.
    \end{equation*}
    Begin with $f(0) = 0$.
    Now, suppose we have defined $f$ on the interval $[0, v_k]$ for some $k\in\N$.
    Let $\theta_k = v_k/u_k$.
    With \cref{e:dimls-e} in mind, we would like to define $f$ on the interval $[v_k, u_k]$ such that, for $\theta_k \leq \theta \leq 1$,
    \begin{equation}\label{e:fdef}
        f(u_k) - f(\theta u_k) = u_k\varphi(\theta).
    \end{equation}
    Recalling that the value $f(v_k)$ was defined in the previous stage of the construction, for $v_k \leq u \leq u_k$ we define
    \begin{equation*}
        f(u) = f(v_k) + u_k\varphi(\theta_k) - u_k\varphi(u/u_k)
    \end{equation*}
    and it is easy to check that \cref{e:fdef} is satisfied since $\varphi(1) = 0$.
    Since $\varphi$ is decreasing and $d$-Lipschitz, $f$ is increasing and $d$-Lipschitz on $[0, u_k]$.
    Finally, for $u_k \leq u \leq v_{k+1}$, we simply define
    \begin{equation}\label{e:fdef-2}
        f(u) = d(u - u_k) + f(u_k).
    \end{equation}
    Therefore, $f\in\mathcal{Y}(d)$.

    It remains to show that
    \begin{equation*}
        \liminf_{u\to\infty}\frac{f(u) - f(\theta u)}{u} = \varphi(\theta)
    \end{equation*}
    for all $0 < \theta \leq 1$.
    Certainly
    \begin{equation*}
        \liminf_{u\to\infty}\frac{f(u) - f(\theta u)}{u} \leq \lim_{k\to\infty}\frac{f(u_k) - f(\theta u_k)}{u_k} = \varphi(\theta)
    \end{equation*}
    since the final equality holds for all $k$ sufficiently large by \cref{e:fdef}.
    To complete the proof, fix $0<\kappa \leq 1$: it suffices to show that for all $u$ sufficiently large,
    \begin{equation}\label{e:upper}
        f(u) - f(\kappa u) \geq u\varphi(\kappa).
    \end{equation}
    Let $N$ be sufficiently large so that for all $k \geq N$, $v_k/u_k \leq\kappa$ and $u_k/v_{k+1} \leq \kappa$, and let $u \geq v_N/\kappa$ be arbitrary.
    The choice of $N$ implies that there are four possible cases depending on the values of $u$ and $\kappa u$:
    \begin{enumerate}
        \item\label{e:c1} \emph{Suppose $v_k \leq \kappa u \leq u \leq u_k$ for some $k\in\N$.}
            Let $u = \theta u_k$ so that $\kappa u = \kappa\theta u_k$.
            Then by \eqref{i:superadd},
            \begin{equation*}
                f(u) - f(\kappa u) = u_k\varphi(\kappa\theta) - u_k \varphi(\theta) \geq u_k\theta\varphi(\kappa) = u\varphi(\kappa).
            \end{equation*}
        \item \emph{Suppose $u_k \leq \kappa u \leq u \leq v_{k+1}$ for some $k\in\N$.}
            Since $\varphi$ is $d$-Lipschitz, $\varphi(\kappa)\leq d(1-\kappa)$ so
            \begin{equation*}
                f(u) - f(\kappa u) = u (1-\kappa)d \geq u\varphi(\kappa).
            \end{equation*}
        \item \emph{Suppose $v_k \leq \kappa u \leq u_k \leq u\leq v_{k+1}$ for some $k\in\N$.}
            Let $\lambda$ be such that $\kappa u = \lambda u_k$ and write $\theta = \kappa/\lambda \leq 1$.
            Then by \eqref{i:conv} (recalling that $\mathcal{U}(d)\subset\mathcal{H}(d)$),
            \begin{align*}
                f(u) - f(\kappa u) &= d(u - u_k) + u_k\varphi(\lambda)\\
                                   &= u(d(1-\theta) + \theta\varphi(\lambda))\\
                                   &\geq u\varphi(\lambda\theta) = u\varphi(\kappa).
            \end{align*}
        \item \emph{Suppose $u_{k-1} \leq \kappa u \leq v_k \leq u \leq u_k$ for some $k\in\N$.}
            Let $\lambda$ be such that $\lambda u = v_k$.
            Applying case \cref{e:c1} on the interval $[\lambda u, u]$ and then using the fact that $\varphi$ is $d$-Lipschitz,
            \begin{align*}
                f(u) - f(\kappa u) &= f(u) - f(\lambda u) + d(\lambda u - \kappa u)\\
                                   & \geq u(\varphi(\lambda) + d(\lambda - \kappa))\\
                                   & \geq u \varphi(\kappa).
            \end{align*}
    \end{enumerate}
    Thus \cref{e:upper} follows, completing the proof.
\end{proof}
To conclude, we note that one can prove similar results as proven in \cref{ss:mono} for the monotone lower spectrum of uniform sets.
Given $\alpha \geq 0$, $\lambda\in(0,1)$, and $0 \leq t \leq \alpha(1-\lambda)$, we define the function
\begin{equation*}
    \psi_{\alpha, \lambda, t}(\theta) = \begin{cases}
        t + \alpha(\lambda - \theta) &: 0 < \theta \leq \lambda,\\*
        t\frac{1-\theta}{1-\lambda} &: \lambda < \theta < 1.
    \end{cases}
\end{equation*}
This is the function which has constant slope $-\alpha$ on the interval $(0,\lambda]$, with value $t$ at $\lambda$, and constant slope on $[\lambda, 1)$ with value $0$ at $1$.
The condition on $t$ ensures that $\psi_{\alpha, \lambda, t}\in\mathcal{M}(\alpha)\cap\mathcal{U}(\alpha)$.
Recalling the definition of $\phi$ from \cref{ss:mono}, we note that $\psi_{\alpha, \lambda,t}\leq\phi_{\alpha,\lambda,t}$, and moreover the inequality is strict on $(0,\lambda)$ if and only if $t < \alpha(1-\lambda)$.
Following an identical approach to that in \cref{ss:mono}, one can obtain the following result.
\begin{proposition}\label{p:extra}
    Let $d\in\N$ and $\varphi\colon(0,1)\to[0,d]$.
    Then the following are equivalent.
    \begin{enumerate}[nl,r]
        \item $\varphi$ is $d$-Lipschitz and satisfies \eqref{i:mono}.
        \item There exists a uniform $F\subset\R^d$ such that $(1-\theta)\dimuLs\theta F = \varphi(\theta)$ for all $\theta\in(0,1)$.
        \item $\varphi$ is an infimum of a family of functions of the form $\psi_{d, \lambda, t}$ for $\lambda\in(0,1)$ and $0\leq t \leq d(1-\lambda)$.
    \end{enumerate}
\end{proposition}

\begin{acknowledgements}
    The authors thank Vilma Orgoványi for comments on a draft version of this paper.
    This collaboration began when H.C.\ visited the University of Jyväskylä, supported by the \emph{JYU Visiting Fellow Programme}.
    He is grateful for the friendly and stimulating research environment there.

	A.B.\ was supported by the Marie Skłodowska-Curie Actions postdoctoral fellowship \emph{FoDeNoF} (grant no.\ 101210409) from the European Union.
    H.C.\ was supported by NSFC 12401109.
    A.R.\ was supported by Tuomas Orponen's grant from the Research Council of Finland via the project \emph{Approximate incidence geometry}, grant no.\ 355453.
    W.W.\ was supported by NSFC 12061086.

    Claude Opus 4.6 was used to write part of the TikZ code for \cref{f:mini} and \cref{f:non-mono}.
    OpenAI Astra was also used for detailed proof-reading of the document.
    AI was not used for anything else in the document.
\end{acknowledgements}
\appendix
\section{Monotonization of the lower spectrum}\label{s:monot}
The following result was proven under the additional assumption that $X$ is doubling and uniformly perfect in \cite[Theorem~1.1]{zbl:1434.28014}, and then with only the doubling assumption in \cite[Theorem~A.2]{zbl:1479.28010}.
Here we dispense with the doubling assumption, and moreover give a simple proof which only uses monotonicity of covering numbers.
\begin{theorem}\label{p:rel}
    Let $X$ be a non-empty metric space.
    Then for all $0<\theta < 1$,
    \begin{equation*}
        \dimuLs\theta X = \inf_{0<\lambda\leq \theta}\dimLs\lambda X.
    \end{equation*}
\end{theorem}
\begin{proof}
    It suffices to prove that
    \begin{equation}\label{e:req}
        \dimuLs\theta X \geq \inf_{0<\lambda \leq \theta}\dimLs\lambda X
    \end{equation}
    since the opposite inequality is immediate from the definition.

    Let $\gamma$ be defined as in the proof of \cref{c:dimL-lim}; that is, for $v \leq u$,
    \begin{equation*}
        \gamma(u,v) = \log \inf_{x\in X} N_{2^{-u}}(B(x, 2^{-v})).
    \end{equation*}
    The only facts about $\gamma$ that we will require are that $\gamma$ is non-negative, $v\mapsto \gamma(u,v)$ is decreasing for all $u$, and (from the proof of \cref{c:dimL-lim}) for $0<\theta < 1$,
    \begin{equation*}
        \dimLs\theta X = \liminf_{u\to\infty}\frac{\gamma(u, \theta u)}{u(1-\theta)}.
    \end{equation*}

    We may assume that $\dimuLs\theta X < \infty$, or else the inequality is trivial.
    Let $t > s > \dimuLs\theta X$ be arbitrary.
    By definition of $\dimuLs\theta X$, for all $n \in\N$ there is a $u_n \geq 0$ and $\lambda_n \in [0, \theta]$ such that
    \begin{equation*}
        \gamma(u_n, \lambda_n u_n) \leq s u_n(1-\lambda_n) - n.
    \end{equation*}
    Since $\gamma$ is non-negative, we must have $\lim_{n\to\infty}u_n = \infty$.
    Passing to a subsequence if necessary, we may assume that $\lim_{n\to\infty}\lambda_n = \lambda\in [0,\theta]$.
    If $\lambda = \theta$, then $\lambda_n \leq \theta$ for all $n\in\N$ so that
    \begin{equation*}
        \gamma(u_n, \theta u_n) \leq \gamma(u_n, \lambda_n u_n) \leq s u_n(1-\lambda_n).
    \end{equation*}
    Since $\theta < 1$, $\lim_{n\to\infty}(1-\lambda_n)^{-1} = (1-\theta)^{-1}$ and therefore
    \begin{equation*}
        \dimLs\theta X \leq \liminf_{n\to\infty} \frac{\gamma(u_n, \theta u_n)}{u_n(1-\lambda_n)} \leq s < t.
    \end{equation*}
    Otherwise, $\lambda < \theta$, and let $\lambda < \kappa \leq \theta$ be such that $s(1-\lambda)/(1-\kappa) < t$.
    Since $\lambda_n \leq \kappa$ for all $n$ sufficiently large,
    \begin{equation*}
        \gamma(u_n, \kappa u_n) \leq \gamma(u_n, \lambda_n u_n) \leq s u_n(1-\lambda_n).
    \end{equation*}
    Therefore, exactly as above,
    \begin{equation*}
        \dimLs\kappa X \leq \liminf_{n\to\infty}\frac{\gamma(u_n, \lambda_n u_n)}{u_n(1-\lambda_n)}\cdot\frac{1-\lambda_n}{1-\kappa} \leq s\frac{1-\lambda}{1-\kappa} \leq t.
    \end{equation*}
    In any case, we have shown that for all $t > \dimuLs\theta X$ there is a $\lambda\in(0, \theta]$ so that $\dimLs\lambda X \leq t$.
    Therefore \cref{e:req} follows.
\end{proof}

\section{Summary of the main inequalities}\label{s:ineq}
In this appendix, we list all of the inequalities used in the various classification theorems and state the key properties along with short proofs.
More detailed versions of these proofs can be found throughout the text in the places where the statement is required.
\begin{definition}
    Let $\alpha \geq 0$ and let $\varphi\colon(0,1)\to [0,\alpha]$ be a function.
    We consider the following inequalities concerning $\varphi$.
    \begin{enumerate}[nl]
        \labelitem{S}{S} \emph{Superadditive:} For all $\lambda,\theta\in(0,1)$,
            \begin{equation*}
                \varphi(\lambda\theta) \geq \varphi(\theta) + \theta\varphi(\lambda).
            \end{equation*}
        \labelitem{W}{W} \emph{Weak $\alpha$-Lipschitz:} For all $\lambda,\theta\in(0,1)$,
            \begin{equation*}
                \varphi(\lambda\theta) \leq (1-\theta) \alpha + \theta\varphi(\lambda).
            \end{equation*}
        \labelitem{M}{M} \emph{Monotone:} The function $\theta\mapsto \varphi(\theta)/(1-\theta)$ is decreasing on $(0,1)$.
        \labelitem{L}{L} \emph{$\alpha$-Lipschitz:} The function $\varphi$ is $\alpha$-Lipschitz.
    \end{enumerate}
\end{definition}
We note from the introduction and \cref{p:extra} that the inequalities relevant for the various classification results are as follows:
\begin{itemize}[nl]
    \item $\theta\mapsto(1-\theta)\dimLs\theta F$ (for general $F$) uses \eqref{S} and \eqref{W}
    \item $\theta\mapsto(1-\theta)\dimLs\theta F$ (for uniform $F$) uses \eqref{S} and \eqref{L}
    \item $\theta\mapsto(1-\theta)\dimuLs\theta F$ (for general $F$) uses \eqref{M} and \eqref{W}
    \item $\theta\mapsto(1-\theta)\dimuLs\theta F$ (for uniform $F$) uses \eqref{M} and \eqref{L}
\end{itemize}
Let us make the following simple observations.
\begin{enumerate}
    \item \eqref{S} implies that $\varphi$ is decreasing, since $\theta\varphi(\lambda) \geq 0$.
        In particular, we may define $\varphi(0) = \lim_{\theta \to 0}\varphi(\theta)$.
    \item\label{o} Iterating \eqref{S} gives
        \begin{equation*}
            \alpha\geq\varphi(\theta^n)\geq\frac{1-\theta^n}{1-\theta}\varphi(\theta).
        \end{equation*}
        Letting $n\to\infty$, we obtain $\varphi(\theta)\leq\alpha(1-\theta)$ and hence $\lim_{\theta\nearrow1}\varphi(\theta)=0$.
        We therefore define $\varphi(1)=0$.
    \item If \eqref{W} holds and $\lim_{t\nearrow1}\varphi(t)=0$, then $\varphi(\theta)\leq\alpha(1-\theta)$.
        Indeed, apply \eqref{W} with parameters $t$ and $\theta/t$, where $\theta<t<1$, and let $t\nearrow1$.
    \item\label{i:wc} \eqref{W} is a weak type of convexity at $0$: for each $\lambda\in(0,1)$, $\varphi$ is bounded above on $(0,\lambda]$ by the line segment joining $(0,\alpha)$ and $(\lambda,\varphi(\lambda))$.
    \item Using \cref{i:wc}, it follows that \eqref{W} holds if and only if the function
        \begin{equation*}
            \theta\mapsto \frac{\alpha - \varphi(\theta)}{\theta}
        \end{equation*}
        is decreasing.
        If $\varphi$ is also decreasing, then $\varphi$ is $\theta^{-1}(\alpha - \varphi(\theta))$-Lipschitz on $[\theta, 1)$ for all $0<\theta<1$.
        Together with boundedness and monotonicity, this gives a unique continuous extension to $[0,1]$.
    \item\label{i:prev} Rearranging \eqref{S}, for all $0<\lambda<1$ and $0<\theta\leq 1$,
        \begin{equation*}
            \frac{\varphi(\lambda\theta) - \varphi(\theta)}{\theta - \lambda\theta} \geq \frac{\varphi(\lambda)}{1-\lambda}.
        \end{equation*}
        The left hand side is the negative of the slope of the line segment from $(\lambda\theta, \varphi(\lambda\theta))$ to $(\theta, \varphi(\theta))$, and the right hand side is the negative of the slope of the line segment from $(\lambda, \varphi(\lambda))$ to $(1, 0)$.
        When $\lambda=1$, \eqref{S} reduces to $\varphi(1)=0$.
    \item \eqref{M} has a geometric interpretation: for any $0<\theta<1$, $\varphi$ is bounded above on $[\theta,1)$ by the line segment joining $(\theta,\varphi(\theta))$ and $(1,0)$.
        In particular, it follows from \cref{i:prev} that \eqref{M} implies \eqref{S}.
    \item If \eqref{L} holds and $\lim_{\theta\nearrow1}\varphi(\theta)=0$, then \eqref{W} holds.
        Indeed, since $\varphi$ is $\alpha$-Lipschitz, $\alpha \lambda \leq \alpha - \varphi(\lambda)$ and therefore
        \begin{equation*}
            \varphi(\lambda\theta) - \varphi(\lambda) \leq \alpha \lambda (1-\theta) \leq (1-\theta)(\alpha - \varphi(\lambda)).
        \end{equation*}
        Rearranging this inequality yields \eqref{W}.
\end{enumerate}
\end{document}